\documentclass[smallextended]{svjour3}       

\smartqed  

\input{frontmatter}

\newcommand{\ud}{\mathrm{d}}

\newcommand{\R}{{\mathbb R}}

\newcommand{\Diff}{\mathrm{Diff}}
\newcommand{\Imm}{\mathrm{Imm}}
\newcommand{\Emb}{\mathrm{Emb}}

\newcommand*\scale{\sigma}

\def\O{\mathcal{O}}

\def\d{\textrm{d}}
\def\LipImm{\mathrm{LipImm}}

\begin{document}

 \maketitle
 \begin{abstract} 
 	The nonlinear spaces of shapes (unparameterized immersed curves or submanifolds) are of interest for many applications in image analysis, such as the identification of  shapes that are  similar
	modulo the action of some group. In this paper we study a general representation of shapes as currents, which are based on linear spaces and are suitable for numerical discretization,
	being robust to noise. We develop the theory of currents for shape spaces by considering both the analytic and numerical  aspects of the problem. In particular, we study the analytical properties of the current map and  the $H^{-s}$ norm that it induces on shapes. We determine the conditions under which the current determines the shape. We then  provide a finite element-based discretization of the currents that is a practical computational tool for shapes. Finally, we demonstrate this approach on a variety of examples.






	
 	\textbf{MSC~2010:} 32U40m 62M40, 65D18, 74S05
 \end{abstract}

 \listoftodos


\section{Introduction}
\label{sec:introduction}
 
``Shape,'' wrote David Mumford in his 2002 ICM address \cite{Mumford2002}, ``is the ultimate nonlinear thing.''
The set of smooth simple closed curves is an example of a shape space; in the study of shape one seeks
analytic and numerical methods to work with such curves: to compare, to classify, to recognise, to evolve---to
understand them. As is usual in mathematics, nonlinear things are constructed out of linear things, namely vector spaces, using simple operations: maps, quotients, and open subsets. 

The aforementioned curves are constructed by first considering parameterised curves $\phi\colon S^1\to\R^2$, and then identifying curves that differ only by a reparameterisation. 
The set of smooth parameterised curves consists of all smooth embeddings $\Emb(S^1,\R^2)$ of $S^1$ into $\R^2$.
A smooth reparameterisation corresponds to composition from the right by a diffeomorphism of $S^1$, i.e., by an element of the set $\Diff(S^1)$.
Thus, the set of (images of) simple closed curves is given by the quotient~$\Emb(S^1,\R^2)/\Diff(S^1)$. 
More generally, let $M$ and $N$ be manifolds of dimension $m$ and $n$ respectively. 
A \emph{shape space} is the set of $m$-dimensional submanifolds of $N$ that are diffeomorphic to $M$, which is realized as $\Emb(M,N)/\Diff(M)$. (This is a nonlinear analogue of the Grassmannian of linear subspaces of a specific dimension.) There are other shape spaces as well, such as the immersed shapes $\Imm(M,N)/\Diff(M)$, shapes in which $M$ has a boundary, piecewise smooth shapes, and oriented shapes.

It is desirable to recognise examples of the same shape, those that are identical up to the action of $\Diff(M)$ and possibly noise or other obfuscation. This can be studied by finding a metric that is blind to changes inside the class, or by finding a representation of the shapes that is invariant under the action of $\Diff(M)$. 
An example of such a representation is the \emph{differential invariant signature} of shapes, an influential new paradigm introduced by \citet{calabi1998differential}. 
For example, to recognize planar curves up to Euclidean transformations a signature curve is $(\kappa,\kappa_s)(S^1)\subset \R^2$, where $\kappa$ is Euclidean curvature and $\kappa_s$ its derivative with respect to arclength. 
This is clearly invariant under Euclidean motions \emph{and} under reparameterisations. 
As Calabi et~al.\ comment, ``The recognition problem includes a comparison principle that would be able to tell whether two signature curves are close in some sense. Thus, we effectively reduce the group-invariant recognition problem to the problem of imposing a `metric' on the space of shapes but now by `shape' we mean the signature curve, not the original object.'' 
Such metrics can be constructed through \emph{shape currents}.

The method of shape currents, first suggested by \citet{Glaunes2008}, is based on embedding the nonlinear shape space in a vector space endowed with a metric, thereby allowing the construction of flexible families of metrics on shape spaces that are easy to compute.
It is robust to noise and provides for control of the resolution and accuracy of the representation.

In previous work, metrics on shape currents have been combined with optimization routines for registration, typically by deforming the shape by left action of a diffeomorphism group (see \autoref{sec:related} below for more details of the use of currents in shape analysis).
Here, we take the viewpoint of Calabi et~al.\ that left symmetries (registration) are taken care of by computing a signature curve, so the only remaining step is to impose a metric on shape space.
This way we obtain a distance on shapes modulo any classical transformation group; see Example~\ref{ex:diffsig}.


The central object in the paper is the \emph{current map} (see \autoref{def:currentmap}),
which takes
a function $\phi\colon M\to N$
and associates it with
\begin{equation}
  \label{eq:currentmap}
[\phi](\alpha) := \int_M \phi^*\alpha
\end{equation}
where $\alpha\in\Lambda^m(N)$ is an $m$-form on $N$. 
The current map is invariant under orientation-preserving reparameterizations (see \autoref{prop:current}).
Continuing from \citet{Glaunes2008}, our aim is, in broad terms, to study this map and the shape distances it induces.
Let us summarize the main results.


We fix a class of functions $\phi$ and a topology on them, namely that of the Lipschitz immersions
$\phi\colon S^1\to\Omega\subset\R^2$, and likewise fix a class of forms $\alpha$ and a topology on them.
Since $[\phi]$ is a linear map from 1-forms to the reals, it is natural to first adopt a Sobolev norm on 1-forms and
then to demand that $[\phi]$ be a continuous map, so that we can adopt the corresponding operator norm 
on currents.

In \autoref{sec:analytical}, we show the following results (see \autoref{def:sobolev} and \autoref{def:lipimm} for the definitions of $H^s\Lambda^1(\Omega)$, $H^{-s}\Lambda^1(\Omega)$ and $\LipImm(S^1,\Omega)$):
\begin{itemize}
\item
  The map $\alpha \mapsto [\phi](\alpha)$ defined in \eqref{eq:currentmap} is continuous and linear in $\alpha$ (\autoref{prop:bounded}), thus showing that $[\phi]$ lies in $H^{-s}\Lambda^1(\Omega)$.
  This generalizes \cite[Proposition~1]{Glaunes2008} to Lipschitz curves.
  \item 
    The current map from $\LipImm(S^1,\Omega)$ to $H^{-s}\Lambda^1(\Omega)$ is Hölder-continuous for $s\ge 1$ (\autoref{prop:cts})
    \item
    The current map (with same domain and codomain as above) is differentiable for $s\ge 2$ (\autoref{prop:diff}).
    This regularity is important: to use invariants to recognize shapes, they must have the property that nearby shapes have nearby invariants.
\item For a given subspace of the space of functions $S^1\to \Omega$,
  and a given linear space of 1-forms on $\Omega$,
  we show that the current map $\phi \mapsto [\phi]$ essentially determines the shape of $\phi$ (\autoref{prop:emb}).
\end{itemize}

In the numerical part (\autoref{sec:finel}), we discretize our construction. 
The currents can be evaluated for all $\alpha$ in a finite element space, which gives
a flexible and general representation of shapes.
Specifically, we introduce
spaces of finite elements $V$ on $M$ and $W$ on $\Lambda^m(N)$ and evaluate the current
$[\phi_V]|_W$, where $\phi_V$ is the approximation of $\phi$ in $V$.
A simple example is shown in \autoref{fig:easy}.
In this case, piecewise constant elements determine a piecewise constant approximant that interpolates the shape at the element edges. 
The norm on currents restricts to $W$ in a natural way, yielding a discretized norm on shapes.

In \autoref{sec:finel}, we show the following results:
\begin{itemize}
\item In \autoref{sec:quadrature}, we demonstrate that the quadrature errors in evaluating the currents are typically small and that the method is robust in the presence of noise.
This robustness stems from the cancellation property of oscillatory integrals.
The currents can be accurately computed even for very rough shapes (not even Lipschitz) and for noisy shapes.
\item
  In \autoref{sec:approx}, we study how accurately the discretized currents determine the shapes.
This is a question in approximation theory. 
\item
In Propositions~\ref{prop:constants} and \ref{prop:approx} we show that 
the order of approximation is 2 for piecewise constant, 3 for piecewise linear, and 5 for piecewise quadratic elements in $W$.
\item
	In \autoref{sec:metric}, we introduce a discretization of the metric on shapes, so that each shape is approximated by a point in a vector space equipped with an Euclidean metric.
\end{itemize}

Finally, we give several numerical examples of the geometry of shape space induced by the discretization in \autoref{sec:examples}. 
The method does not only compare pairs of shapes, it provides a direct approximation of shape space and its geometry: we present numerical experiments (e.g., Examples~\ref{ex:rounded} and~\ref{ex:final}) 
applying Principal Components Analysis directly to the current representation  in order to successfully separate classes of shapes.


\begin{figure}
\centering
\begin{tikzpicture}[scale=0.5]
	\begin{pgfonlayer}{nodelayer}
		\node [style=none] (0) at (-5, 0) {};
		\node [style=none] (1) at (0, 3.25) {};
		\node [style=none] (2) at (2, 1.25) {};
		\node [style=none] (3) at (2, -1.25) {};
		\node [style=none] (4) at (1.25, -2.75) {};
	\end{pgfonlayer}
	\begin{pgfonlayer}{edgelayer}
		\draw [very thick, color=red, <-, bend left, looseness=0.75] (1.center) to (2.center);
		\draw [very thick, color=red, <-, in=75, out=-60, looseness=0.75] (2.center) to (3.center);
		\draw [very thick, color=red, bend left=15] (3.center) to (4.center);
		\draw [very thick, color=red, <-, in=260, out=225] (4.center) to (0.center);
		\draw [very thick, color=red, <-, bend left=45, looseness=0.75] (0.center) to (1.center);
	\end{pgfonlayer}
\end{tikzpicture}
\begin{tikzpicture}[scale=0.5]
	\begin{pgfonlayer}{nodelayer}
		\node [style=none] (0) at (-5, 0) {};
		\node [style=none] (1) at (0, 3.25) {};
		\node [style=none] (2) at (2, 1.25) {};
		\node [style=none] (3) at (2, -1.25) {};
		\node [style=none] (4) at (1.25, -2.75) {};
		\node [style=none] (5) at (-5.75, 0) {};
		\node [style=none] (6) at (-4, 0) {};
		\node [style=none] (7) at (-5, 1.75) {};
		\node [style=none] (8) at (-3, 1.25) {};
		\node [style=none] (9) at (-3, 3.25) {};
		\node [style=none] (10) at (-1.75, 2.25) {};
		\node [style=none] (11) at (-0.75, 4.25) {};
		\node [style=none] (12) at (-0.25, 2.5) {};
		\node [style=none] (13) at (2, 3) {};
		\node [style=none] (14) at (0.75, 1.25) {};
		\node [style=none] (15) at (3, 0.75) {};
		\node [style=none] (16) at (0.75, -0.75) {};
		\node [style=none] (17) at (-0.25, -2.5) {};
		\node [style=none] (18) at (-2, -4) {};
		\node [style=none] (19) at (-2.25, -2.25) {};
		\node [style=none] (20) at (-4.75, -2.75) {};
		\node [style=none] (21) at (-0.5, 0.25) {};
		\node [style=none] (22) at (-1.75, -0.5) {};
		\node [style=none] (23) at (2.5, -1) {};
		\node [style=none] (24) at (1.5, -3.5) {};
	\end{pgfonlayer}
	\begin{pgfonlayer}{edgelayer}
		\draw [very thick, color=red, <-, bend left, looseness=0.75] (1.center) to (2.center);
		\draw [very thick, color=red, <-, in=75, out=-60, looseness=0.75] (2.center) to (3.center);
		\draw [very thick, color=red, bend left=15] (3.center) to (4.center);
		\draw [very thick, color=red, <-, in=260, out=225] (4.center) to (0.center);
		\draw [very thick, color=red, <-, bend left=45, looseness=0.75] (0.center) to (1.center);
		\draw (5.center) to (7.center);
		\draw (7.center) to (6.center);
		\draw (6.center) to (5.center);
		\draw (7.center) to (8.center);
		\draw (8.center) to (6.center);
		\draw (9.center) to (8.center);
		\draw (9.center) to (7.center);
		\draw (9.center) to (10.center);
		\draw (10.center) to (8.center);
		\draw (10.center) to (11.center);
		\draw (11.center) to (9.center);
		\draw (10.center) to (12.center);
		\draw (12.center) to (11.center);
		\draw (12.center) to (13.center);
		\draw (13.center) to (11.center);
		\draw (14.center) to (13.center);
		\draw (12.center) to (14.center);
		\draw (14.center) to (15.center);
		\draw (15.center) to (13.center);
		\draw (23.center) to (15.center);
		\draw (14.center) to (23.center);
		\draw (16.center) to (23.center);
		\draw (16.center) to (14.center);
		\draw (16.center) to (24.center);
		\draw (24.center) to (23.center);
		\draw (16.center) to (17.center);
		\draw (17.center) to (24.center);
		\draw (24.center) to (18.center);
		\draw (18.center) to (17.center);
		\draw (19.center) to (18.center);
		\draw (19.center) to (17.center);
		\draw (19.center) to (20.center);
		\draw (20.center) to (18.center);
		\draw (6.center) to (19.center);
		\draw (6.center) to (20.center);
		\draw (20.center) to (5.center);
		\draw (10.center) to (14.center);
		\draw  (14.center) to (21.center);
		\draw (21.center) to (10.center);
		\draw  (22.center) to (21.center);
		\draw  (8.center) to (22.center);
		\draw  (22.center) to (6.center);
		\draw  (22.center) to (19.center);
		\draw  (22.center) to (17.center);
		\draw  (22.center) to (16.center);
		\draw  (16.center) to (21.center);
		\draw  (10.center) to (22.center);
	\end{pgfonlayer}
	
	\end{tikzpicture}
	\begin{tikzpicture}[scale=0.5]
	\begin{pgfonlayer}{nodelayer}
		\node [fill=blue, circle, style=none, minimum size=1.0 mm] (0) at (-5, 0) {};
		\node [style=none] (1) at (0, 3.25) {};
		\node [style=none] (2) at (2, 1.25) {};
		\node [style=none] (3) at (2, -1.25) {};
		\node [style=none] (4) at (1.25, -2.75) {};
		\node [style=none] (5) at (-5.75, 0) {};
		\node [style=none] (6) at (-4, 0) {};
		\node [style=none] (7) at (-5, 1.75) {};
		\node [style=none] (8) at (-3, 1.25) {};
		\node [style=none] (9) at (-3, 3.25) {};
		\node [style=none] (10) at (-1.75, 2.25) {};
		\node [style=none] (11) at (-0.75, 4.25) {};
		\node [style=none] (12) at (-0.25, 2.5) {};
		\node [style=none] (13) at (2, 3) {};
		\node [style=none] (14) at (0.75, 1.25) {};
		\node [style=none] (15) at (3, 0.75) {};
		\node [style=none] (16) at (0.75, -0.75) {};
		\node [style=none] (17) at (-0.25, -2.5) {};
		\node [style=none] (18) at (-2, -4) {};
		\node [style=none] (19) at (-2.25, -2.25) {};
		\node [style=none] (20) at (-4.75, -2.75) {};
		\node [style=none] (21) at (-0.5, 0.25) {};
		\node [style=none] (22) at (-1.75, -0.5) {};
		\node [style=none] (23) at (2.5, -1) {};
		\node [style=none] (24) at (1.5, -3.5) {};
		\node [fill=blue, circle, style=none, minimum size=1.0 mm] (25) at (2.15, 0.92) {};
		\node [fill=blue, circle, style=none, minimum size=1.0 mm] (26) at (2.15, -0.55) {};
		\node [fill=blue, circle, style=none, minimum size=1.0 mm] (27) at (2.07, -0.95) {};
		\node [style=none] (28) at (1.25, -2.75) {};
		\node [fill=blue, circle, style=none, minimum size=1.0 mm] (29) at (1.28, -2.7) {};
		\node [fill=blue, circle, style=none, minimum size=1.0 mm] (30) at (0.8, -3.1) {};
		\node [fill=blue, circle, style=none, minimum size=1.0 mm] (31) at (-1.375, -3.46) {};
		\node [fill=blue, circle, style=none, minimum size=1.0 mm] (32) at (-2.1, -3.31) {};
		\node [fill=blue, circle, style=none, minimum size=1.0 mm] (33) at (-3.72, -2.53) {};
		\node [fill=blue, circle, style=none, minimum size=1.0 mm] (34) at (-4.5, -1.8) {};
		\node [fill=blue, circle, style=none, minimum size=1.0 mm] (35) at (-4.57, 0.98) {};
		\node [fill=blue, circle, style=none, minimum size=1.0 mm] (36) at (-4.1, 1.54) {};
		\node [fill=blue, circle, style=none, minimum size=1.0 mm] (37) at (-3, 2.4) {};
		\node [fill=blue, circle, style=none, minimum size=1.0 mm] (38) at (-2.375, 2.8) {};
		\node [fill=blue, circle, style=none, minimum size=1.0 mm] (39) at (-1.3, 3.2) {};
		\node [fill=blue, circle, style=none, minimum size=1.0 mm] (40) at (-0.45, 3.3) {};
		\node [fill=blue, circle, style=none, minimum size=1.0 mm] (41) at (0.8725, 2.75) {};
		\node [fill=blue, circle, style=none, minimum size=1.0 mm] (42) at (1.45, 2.2) {};
	\end{pgfonlayer}
	\begin{pgfonlayer}{edgelayer}
		\draw (5.center) to (7.center);
		\draw (7.center) to (6.center);
		\draw (6.center) to (5.center);
		\draw (7.center) to (8.center);
		\draw (8.center) to (6.center);
		\draw (9.center) to (8.center);
		\draw (9.center) to (7.center);
		\draw (9.center) to (10.center);
		\draw (10.center) to (8.center);
		\draw (10.center) to (11.center);
		\draw (11.center) to (9.center);
		\draw (10.center) to (12.center);
		\draw (12.center) to (11.center);
		\draw (12.center) to (13.center);
		\draw (13.center) to (11.center);
		\draw (14.center) to (13.center);
		\draw (12.center) to (14.center);
		\draw (14.center) to (15.center);
		\draw (15.center) to (13.center);
		\draw (23.center) to (15.center);
		\draw (14.center) to (23.center);
		\draw (16.center) to (23.center);
		\draw (16.center) to (14.center);
		\draw (16.center) to (24.center);
		\draw (24.center) to (23.center);
		\draw (16.center) to (17.center);
		\draw (17.center) to (24.center);
		\draw (24.center) to (18.center);
		\draw (18.center) to (17.center);
		\draw (19.center) to (18.center);
		\draw (19.center) to (17.center);
		\draw (19.center) to (20.center);
		\draw (20.center) to (18.center);
		\draw (6.center) to (19.center);
		\draw (6.center) to (20.center);
		\draw (20.center) to (5.center);
		\draw [dashed] (10.center) to (14.center);
		\draw [dashed] (14.center) to (21.center);
		\draw [dashed] (21.center) to (10.center);
		\draw [dashed] (22.center) to (21.center);
		\draw [dashed] (8.center) to (22.center);
		\draw [dashed] (22.center) to (6.center);
		\draw [dashed] (22.center) to (19.center);
		\draw [dashed] (22.center) to (17.center);
		\draw [dashed] (22.center) to (16.center);
		\draw [dashed] (16.center) to (21.center);
		\draw [dashed] (10.center) to (22.center);
		\draw [color=blue] (41.center) to (40.center);
		\draw [color=blue] (40.center) to (39.center);
		\draw [color=blue] (39.center) to (38.center);
		\draw [color=blue] (38.center) to (37.center);
		\draw [color=blue] (37.center) to (36.center);
		\draw [color=blue] (36.center) to (35.center);
		\draw [color=blue] (35.center) to (0.center);
		\draw [color=blue] (0.center) to (34.center);
		\draw [color=blue] (34.center) to (33.center);
		\draw [color=blue] (33.center) to (32.center);
		\draw [color=blue] (32.center) to (31.center);
		\draw [color=blue] (31.center) to (30.center);
		\draw [color=blue] (30.center) to (29.center);
		\draw [color=blue] (29.center) to (27.center);
		\draw [color=blue] (27.center) to (26.center);
		\draw [color=blue] (26.center) to (25.center);
		\draw [color=blue] (25.center) to (42.center);
		\draw [color=blue] (42.center) to (41.center);
	\end{pgfonlayer}
\end{tikzpicture}
\caption{\label{fig:easy}A simple example of finite element currents using piecewise constant elements. 
An oriented shape (the image
of a simple closed curve $\phi$) is given in red on the left. A triangular mesh is laid over the domain.
For each triangle $T$ in the mesh, the finite element currents $\int_{\phi(S^1)\cap T} \d x $ and $\int_{\phi(S^1)\cap T} \d y $ 
are computed numerically. These determine the triangles (shown as solid black lines on the right) that intersect the shape,
and the $x$- and $y$-extents of the shape on those triangles. This data
determines a piecewise-linear approximant that interpolates the shape and the element edges. 
Higher-order currents, such as $\int_{\phi(S^1)\cap T} y\, \d x $ are also possible, and give a more accurate
representation of the shape. See also Figures \ref{fig:2segments}, \ref{fig:pc}.
}
\end{figure}
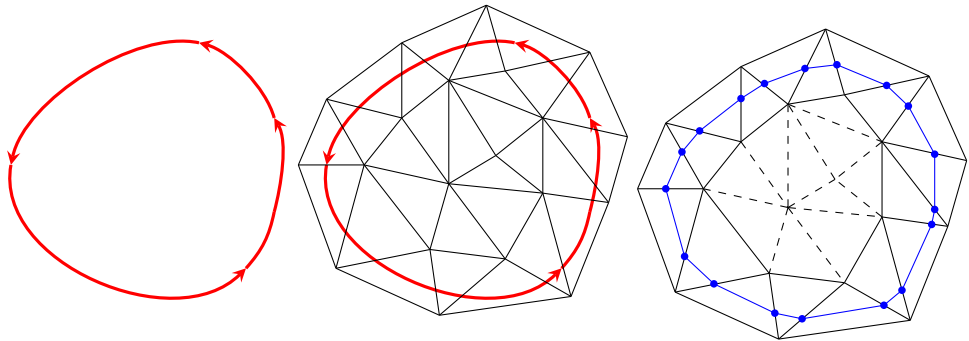

\subsection{Related Literature}
\label{sec:related}

Currents have already been used in shape analysis, primarily for curve and surface matching. 
In \cite{Vaillant2005} a surface in 3D was represented with currents defined on a surface mesh, and a norm was computed to enable the matching of two surfaces using the currents. 
A framework is developed to allow shape currents to be matched in the spirit of the Large Deformation Diffeomorphic Metric Mapping (LDDMM) framework \cite{Beg05}, and the method was demonstrated on surfaces representing shapes and hippocampi. 
A variation on this approach for curves rather than surfaces was developed in \cite{Glaunes2008} (where the currents are referred to as vector-valued measures). 
Again, the aim is a matching algorithm where curves can be deformed onto each other.

In \cite{Durrleman2009} a different benefit of the linear representation provided by currents is recognised, which is that they provide a useful space in which to perform statistical analysis of the deformations between shapes. This was originally considered in \cite{Glaunes2006}, but there is a difficulty that the mean (template) shape has to be defined in such a way that both the shape and deformation are amenable to statistical analysis. A Matching Pursuit algorithm is defined in \cite{Durrleman2009} to provide a computationally tractable representation of shape currents as the number of shapes grows.


The use of currents for matching was extended by Charon and Trouv\'{e} in \cite{Charon2014} as functional currents, where a function is added to the current representation so that the deformation of a shape and some function defined on its surface can be considered simultaneously. The same authors also considered how to deal with cases where the orientation property of currents is undesirable. The fact that a large spike in the appearance of the shape cancels out in the current representation as the positive and negative contributions are virtually identical is a benefit when considering the currents for dealing with noisy representations of shapes. However, in cases where these spikes can truly exist, or where there is orientation information in the image, but it can differ arbitrarily by sign, such as in Diffusion Tensor Images of the brain, the oriented manifold is a disadvantage. This leads to the consideration of varifolds in~\cite{Charon2013}, where the registration of some directed surfaces is demonstrated.  

While there are many other computational approaches to shape space, as far as we are aware they all involve determining a point correspondence between shapes via optimization and/or working directly in the nonlinear shape space; see for example \cite{celledoni2016shape,bauer2014constructing}.

In contrast to the aforementioned work, we focus here solely on shape currents as a way to induce distances and compute statistics on $\Emb(M,N)/\Diff(M)$; we assume that registration (left matching) has already been taken care of, for example through the signature curve, as suggested by \citet{calabi1998differential}.






\section{Currents and their induced metric on shapes}
\label{sec:analytical}

For any natural number $m$, we denote by $\Lambda^m(N)$ the space of smooth $m$-forms on a manifold $N$.
An $m$-\emph{current} is an element of $\Lambda^m(N)^*$, the topological dual of $\Lambda^m(N)$.
Later, in \autoref{sec:contdiff}, we will introduce another topology on $\Lambda^m(N)$,
and we will still call \emph{a current} an element of the corresponding dual.

Currents were introduced by de Rham \cite{de1973varietes}.
They are natural generalisations (or completions) of the pairing by integration between an $m$~dimensional submanifold $M\subset N$ and smooth $m$--forms on $N$.
Currents are instrumental in geometric measure theory, where they are used to
study a very wide class of (not necessarily smooth) subsets of $\R^n$, for example in minimal surface problems \cite{morgan}.
In that field, the functions $\phi$ are typically Lipschitz and the differential forms are smooth.
We allow nonsmooth shapes such as those represented by Lipschitz functions, but as we want to discretize the differential forms by finite elements, we will allow the forms to be nonsmooth as well.

\subsection{De Rham currents and their invariance}\label{sec:drc}

Our parameterization invariant map (\autoref{def:currentmap}) is motivated by the signed area $\int_{\phi(S^1)} y\, \d x $ enclosed by a closed curve.
This is induced from a 1-form ($y\, \d x $) on $\R^2$.\footnote{This also happen to be Euclidean-invariant, but this is \emph{not} relevant to the sequel.}
Currents are invariant under orientation-preserving diffeomorphisms of $M$---sense-preserving
reparameterizations in the case of curves---and hence can be used to factor
out the group $\Diff^+(M)$, the subgroup of orientation-preserving diffeomorphisms.
As can be seen by considering the length and area of a  noisy curve (see the example in \autoref{fig:errorsimple}), currents are very robust to noise.

Although we will mostly focus in the sequel on the case of 
$M=S^1$ and $N=\Omega$,  where $\Omega$ is a domain in $\R^2$,
we formulate the definition of the current map in a more general setting:
\begin{definition}
\label{def:currentmap}
Let $M$ and $N$ be oriented manifolds of dimensions $m$ and $n$,
with $m \leq n$.
We denote by
 $\Imm(M,N)$ the space of immersions from $M$ to $N$.
Let $\phi\colon M\to N$ be an immersion.
We define $[\phi]$ to be the linear function on forms given by
\begin{equation}\label{eq:current_invariant}
	[\phi]\colon \Lambda^m(N)\to \R,\quad [\phi](\alpha) := \int_M \phi^*\alpha = \int_{\phi(M)}\alpha.	
\end{equation}
The \emph{current map} is the corresponding map from parameterized manifolds to currents:
\begin{equation}
  [\cdot] \colon \Imm(M,N) \to \Lambda^m(N)^* .
\end{equation}
\end{definition}

\begin{proposition} 
\label{prop:current}
The current map defined in \autoref{def:currentmap} is invariant with respect to orientation-preserving reparameterizations, that is
$[\phi\circ\psi] = [\phi]$ for all orientation-preserving diffeomorphisms $\psi$ of $M$.
In other words, the current map $\phi \mapsto [\phi]$ is $\Diff^+(M)$-invariant (where $\Diff^+(M)$ denotes the space of orientation-preserving diffeomorphisms on $M$), and it induces a map
from $\Imm(M,N)/\Diff^+(M)$ to $ \Lambda^m(N)^*$.
\end{proposition}
\begin{proof}
This is just the statement that integration of forms is well-defined, i.e., independent of the choice
of coordinates. Specifically,
$$\int_M(\phi\circ\psi)^*\alpha = \int_M \psi^*\phi^*\alpha
= \int_{\psi(M)} \phi^*\alpha = \int_M \phi^*\alpha$$
for orientation-preserving $\psi$. 
\qed
\end{proof}

Currents measure some particular aspects of curves.
Consider the case of planar curves, and in particular the curve $\phi(t)=(\cos(t),0)$ for $0\le t\le 2\pi$. 
This shape retraces itself in opposite directions, so $[\phi](\alpha)=0$ for all $\alpha$; currents cannot distinguish this curve from the 0 curve.

In the case that $\phi(M)$ is a submanifold of $N$, $[\phi]$ is the $m$-current of integration on $\phi(M)$.


\subsection{Properties of the current map}
\label{sec:contdiff}

Before choosing a particular norm in the space of forms, let us motivate our choice.
For the current map to be defined (\autoref{def:currentmap}), we need to be able to take \emph{traces} of a form on any curves (see \autoref{rk:trace}).
In particular, we want this operation to be continuous (see \autoref{prop:bounded}).
This is related to the definition of a \emph{reproducing kernel Hilbert space}, in which traces on \emph{points} are supposed to exist.
Notice however that our setting is strictly more general, as, for instance, $H^1(\R^2)$ is not a reproducing kernel Hilbert space, but traces on curves are well defined.

So,
for its metric and computational
aspects,
we adopt the following $H^s$ Sobolev inner product on 1-forms:
\begin{definition}
  \label{def:sobolev}
 Given a non-negative integer $s$, a {scale parameter} $\scale$, and a domain $\Omega$ of $\R^n$, we denote by $H_{\scale}^s\Lambda^1(\Omega)$ the Hilbert space of forms with scalar product defined by
\begin{equation}
(\alpha,\beta)_{H^s_{\scale}} := \int_\Omega \sum_{i=1,\,2\atop 0\le |k|\le s}\scale^{2|k|}{s\choose k} (D^k\alpha_i)(D^k \beta_i)\, {\textrm d}x_1 {\textrm d}x_2
\end{equation} 
  where $\alpha$ and $\beta$ are 1-forms with coordinates $\alpha = \sum_i\alpha_i \d x_i$ and $\beta = \sum_i \beta_i \d x_i$.
  The dual of that space will be called \emph{$H_{\scale}^{-s}$ currents}, and denoted by $H_{\scale}^{-s}\Lambda^1(\Omega)$.
  We will often omit the scale parameter $\scale$ in the sequel for the sake of readability.
\end{definition}

To understand why $\scale$ is called a scale parameter, suppose that the domain $\Omega = \R^n$ for the moment.
Then to any positive scalar $\lambda$ we can associate the scaling $\lambda \cdot x := \lambda x$.
This function $(\lambda \cdot)$ acts on 1-forms by pull-back, that is (with a slight abuse of notation) $\lambda \cdot \alpha := (\lambda^{-1} \cdot)^* \alpha$.
The resulting form is $\lambda \cdot \alpha (x) = \lambda^{-1} \sum_{i=1,2} \alpha_i(\lambda^{-1} x_1, \lambda^{-1} x_2) \d x_i$,
so we see that $(\lambda \cdot \alpha, \lambda \cdot \beta)_{H^s_{\scale}} = \lambda^{-1} (\alpha, \beta)_{H^s_{\scale/\lambda}}$.
We will also see in \autoref{sec:representers} that $\scale$ determines the length scale at which distances between shapes are be measured.

The properties of the map from curves to $H^{-s}$ currents depends not only on the topology of the currents, but also on that of the curves.
We wish to allow very large classes of curves---indeed, one of the strengths of currents is that they do allow this.
We thus define the following space of curves.
\begin{definition}
\label{def:lipimm}
We define the space of \emph{Lipschitz immersions}, that is, curves $\phi\colon S^1\to\Omega$ such that the components of
$\phi$ are Lipschitz and the tangent vector $\phi'$, wherever it is defined (which is almost everywhere), is nonzero.
We call the space of such curves $\LipImm(S^1,\Omega)$.
\end{definition}


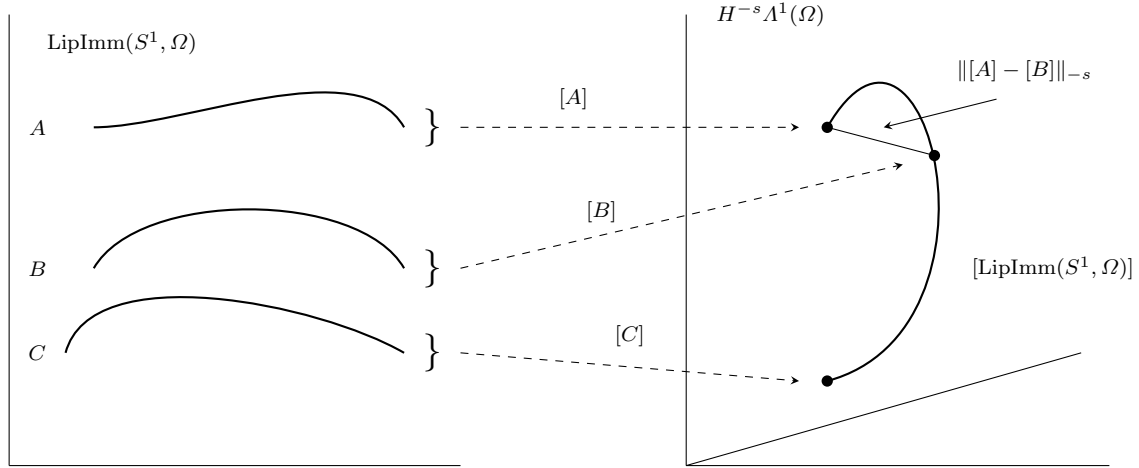
\begin{figure}
\centering
\begin{tikzpicture}[scale=1.5]
	\begin{pgfonlayer}{nodelayer}
		\node [style=none] (0) at (-5, 4) {};
		\node [style=none] (1) at (-5, 0) {};
		\node [style=none] (2) at (-1, 0) {};
		\node [style=none] (3) at (1, 4) {};
		\node [style=none] (4) at (1, 0) {};
		\node [style=none] (5) at (4.75, 0) {};
		\node [style=none] (6) at (4.5, 1) {};
		\node [style=none] (7) at (-1.5, 3) {};
		\node [style=none] (8) at (-4.25, 3) {};
		\node [style=none] (9) at (-1.5, 1.75) {};
		\node [style=none] (10) at (-3.5, 2.25) {};
		\node [style=none] (11) at (-4.25, 1.75) {};
		\node [style=none] (12) at (-1.5, 1) {};
		\node [style=none] (13) at (-2.5, 1.25) {};
		\node [style=none] (14) at (-4.5, 1) {};
		\node [fill=black, circle, minimum size=1.5 mm, style=none] (15) at (2.25, 0.75) {};
		\node [minimum size=1.5 mm, circle, fill=black, style=none] (16) at (2.25, 3) {};
		\node [fill=black, circle, minimum size=1.5 mm, style=none] (17) at (3.2, 2.75) {};
		\node [style=none] (18) at (-4, 3.75) {$\mathrm{LipImm}(S^1,\Omega)$};
		\node [style=none] (19) at (1.75, 4) {$H^{-s}\Lambda^1(\Omega)$};
		\node [style=none] (20) at (-1.25, 3) {\Large $\rbrace$};
		\node [style=none] (21) at (-1.25, 1.75) {\Large $\rbrace$};
		\node [style=none] (22) at (-1.25, 1) {\Large $\rbrace$};
		\node [style=none] (23) at (-1, 3) {};
		\node [style=none] (24) at (-1, 1.75) {};
		\node [style=none] (25) at (-1, 1) {};
		\node [style=none] (26) at (2, 3) {};
		\node [style=none] (27) at (2.92, 2.67) {};
		\node [style=none] (28) at (2, 0.75) {};
		\node [style=none] (29) at (-4.75, 3) {$A$};
		\node [style=none] (30) at (-4.75, 1.75) {$B$};
		\node [style=none] (31) at (-4.75, 1) {$C$};
		\node [style=none] (32) at (0, 3.25) {$[A]$};
		\node [style=none] (33) at (0.25, 2.25) {$[B]$};
		\node [style=none] (34) at (0.5, 1.15) {$[C]$};
		\node [style=none] (35) at (2.75, 3) {};
		\node [style=none] (36) at (3.75, 3.25) {};
		\node [style=none] (37) at (4, 3.5) {$\|[A]-[B]\|_{-s}$};
		\node [style=none] (38) at (4.25, 1.75) {$[\mathrm{LipImm}(S^1,\Omega)]$};
	\end{pgfonlayer}
	\begin{pgfonlayer}{edgelayer}
		\draw (0.center) to (1.center);
		\draw (1.center) to (2.center);
		\draw (3.center) to (4.center);
		\draw (4.center) to (5.center);
		\draw (4.center) to (6.center);
		\draw [thick, bend left=60, looseness=0.75] (11.center) to (9.center);
		\draw [thick, in=150, out=75, looseness=0.75] (14.center) to (12.center);
		\draw [thick, in=120, out=0, looseness=0.75] (8.center) to (7.center);
		\draw [thick, in=15, out=60, looseness=2.00] (16.center) to (15.center);
		\draw [->, dashed] (23.center) to (26.center);
		\draw [->, dashed] (24.center) to (27.center);
		\draw [->, dashed] (25.center) to (28.center);
		\draw (16.center) to (17.center);
		\draw [->] (36.center) to (35.center);
	\end{pgfonlayer}
\end{tikzpicture}
\caption{\label{fig:main}Schematic representation of the induced distance on shapes.
The left-hand side shows the vector space of Lipschitz immersed curves. It is
partitioned into equivalence classes such as $A$, $B$, and $C$, where curves are equivalent if they are related
by a sense-preserving reparameterization. Each equivalence class maps under 
the current map $[\phi]$ to a single point on the right-hand side, which
shows the vector space of linear forms on 1-forms equipped with
the operator norm induced by the $H^s$-metric on 1-forms. The set of
all Lipschitz immersed curves maps into a very small subset of $H^{-s}\Lambda^1(\Omega)$.
This subset is labelled $[\mathrm{LipImm}(S^1,\Omega)]$ on the right.
The distance between two shapes is measured by the `straight line distance'
in the normed vector space $H^{-s}\Lambda^1(\Omega)$.
}
\end{figure}

\begin{remark}
  \label{rk:trace}
  Given a bounded domain $\Omega$ in $\R^n$, with Lipschitz boundary $\partial\Omega$, a typical function $f$ in $H^{1}(\Omega)$ is not continuous and is only defined almost everywhere in $\Omega$.
  Moreover, $\partial\Omega$ has $n$-dimensional Lebesgue measure zero; hence there is no direct meaning we can give to the expression ``$f$ restricted to $\partial\Omega$''.
  The notion of a \emph{trace operator} and the trace theorem resolves this issue for us \cite{Ad2003}.
  More generally, passing from functions in $H^{s}(\Omega)$ to their traces on surfaces of codimension 1 results in a loss of smoothness corresponding to half a derivative.
\end{remark}

The following result generalizes \cite[Proposition 1]{Glaunes2008}, from piecewise $C^1$ to Lipschitz immersions.

\begin{proposition}
\label{prop:bounded}
Let $s\ge 1$ and let $\phi\colon S^1\to\Omega$ be a Lipschitz immersion. Then $[\phi]\colon H^s\Lambda^1(\Omega)\to \R$ is
a bounded linear operator.
\end{proposition}
\begin{proof}
Each Lipschitz immersion $\phi$ can be written as:
\begin{equation}
\phi=\left(\phi_{1},\phi_{2}\right),\label{eq:Imm}
\end{equation}
where $\phi_{1}$ and $\phi_{2}$ are Lipschitz functions from $S^{1}$ to
$\mathbb{R}$. 

Parameterizing $S^{1}$ by  $t\in[0,1)=I$
and letting $\left\{ U_{i}\right\} $ denote a finite system of open
sets covering $I$,  functions $\phi_{1}$ and $\phi_{2}$ (\ref{eq:Imm})
are Lipschitz functions from $U_{i}$ to $\mathbb{R}$.
We recall \cite{Ev} that $\phi_1|_U$ and $\phi_2|_U$ are
Lipschitz if and only if they belong to $W^{1,\infty}(U)$, the space of functions with
essentially bounded first weak derivative. 

Let $(a,b)=V\in\left\{ U_{i}\right\} $ 
and $\alpha=\alpha_{1}\, \d x+\alpha_{2}\, \d y\in H^{s}\Lambda^1(\R^2)$. In coordinates
we have:
\[
\left[\phi|_V\right](\alpha)=\int_{V}(\alpha_{1}\circ\phi\cdot\phi_{1}'+\alpha_{2}\circ\phi\cdot\phi_{2}' ) \, \d t.
\]
Since $\phi_{1}|_V$ and $\phi_{2}|_V$ both belong to $W^{1,\infty}(V)$, and $\alpha_{1}\circ\phi$,
$\alpha_{2}\circ\phi\in L^{2}(\phi(V))$, by the trace theorem (see \autoref{rk:trace}), we obtain:
\[
\left|\left[\phi|_V\right](\alpha)\right|\leq(b-a)\left(\left\Vert \alpha_{1}\right\Vert _{L^{2}(\phi(S^{1}))}\left\Vert \phi_{1}\right\Vert _{W^{1,\infty}(V)}+\left\Vert \alpha_2\right\Vert _{L^{2}(\phi(S^{1}))}\left\Vert \phi_{2}\right\Vert _{W^{1,\infty}(V)}\right)=C<\infty.
\]
Since $I$ is compact we can find a constant $\tilde{C}$ which works
for all sets $U_{i}$, so $[\phi]$ is bounded.
\qed
\end{proof}

\autoref{prop:bounded} implies that the dual (operator) norm
$$ \|[\phi]\|_{H^{-s}\Lambda^1(\Omega)} := \sup_{\alpha\in H^s\Lambda^1(\Omega)\atop\|\alpha\|_{H^s}=1} [\phi](\alpha)$$
is well defined. That is, $[\phi]$ is an element of the Sobolev dual $H^{-s}\Lambda^1(\Omega)$.
We  measure the similarity of shapes by their distance in this dual: for
two shapes $\phi_1$, $\phi_2$, their distance is (see Figure \ref{fig:main} for a pictorial version of this):
$$ d(\phi_1,\phi_2) := \| [\phi_1] - [\phi_2]\|_{-s} := \| [\phi_1] - [\phi_2] \|_{H^{-s}\Lambda^1(\Omega)}.$$

To sum up,  currents map shapes into a Hilbert space (in a highly nonlinear way) and we
measure the distance between shapes using the norm on that Hilbert space.
At first sight this appears to be a wasteful
representation, as it uses functions on the higher dimensional space $\Omega\subset\R^2$ instead of on $S^1$. 
However, in finite dimensional examples in which quotient spaces are represented using invariants, it is common that large numbers of invariants are needed\footnote{For example, when $S^1$ acts
on ${\mathbb C}^n$ by $z_i \mapsto e^{i\theta z_i}$, the set of  invariants
$\bar z_i z_j$, $1\le i,j\le n$---$n^2$ real invariants in all---is complete, and $n^2$ is much larger than $\mathrm{dim}({\mathbb C}^n/S^1)=2n-1$. One can find smaller complete sets, limited by the dimension of the smallest Euclidean space into which ${\mathbb C}^n/S^1$ can be embedded. However, in such sets the individual invariants are more complicated \cite{bandeira2014saving}, so that the total complexity is $\mathcal{O}(n^2)$.}.
Furthermore, this approach allows one to represent much larger classes of objects than
the smooth embeddings, such as weighted, nonsmooth, and immersed shapes. 

As remarked earlier, it is vital that the map from curves to $H^{-s}$ currents be continuous, so that nearby shapes have nearby $H^{-s}$ currents. 
The topology of Lipschitz immersions
is the direct product of $W^{1,\infty}$ in each component. That is, two Lipschitz
immersions $\phi$, $\psi$ are close if the suprema of $|\phi_1(t)-\psi_1(t)|$, $|\phi_2(t)-\psi_2(t)|$,
$|\phi'_1(t)-\psi'_1(t)|$, and $|\phi'_2(t)-\psi'_2(t)|$ over $0\le t<1$ are all small.

\begin{proposition}
\label{prop:cts}
Let $s\ge 1$. Then the current map $[\cdot]\colon \LipImm(S^1,\Omega)\to H^{-s}\Lambda^1(\Omega)$ is Hölder-continuous with exponent $1/2$;
in particular, it is continuous.
\end{proposition}

\begin{proof}
It suffices to prove the proposition for $s=1$, for if two linear functions are close in $H^{-1}$ then they
are close in $H^{-s}$ for $s>1$. First consider the case of embeddings.
Let $\phi$ be a fixed Lipschitz embedding and let $\psi$ be any Lipschitz embedding with $\|\phi_1-\psi_1\|_{W^{1,\infty}(S^1)}<\delta$,
$\|\phi_2-\psi_2\|_{W^{1,\infty}(S^1)}<\delta$.
Let $\alpha\in H^1\Lambda^1(\Omega)$. We need to estimate
\[ A := \left| [\phi](\alpha) - [\psi](\alpha)\right|.  \]
First, we have:
\[A = \left| \int_{\phi(S^1)}\alpha - \int_{\psi(S^1)}\alpha\right|.\]
Let $R$ be the region enclosed between $\phi(S^1)$ and $\psi(S^1)$ (see \autoref{fig:proof}).
From Stokes's theorem:
\begin{equation}
\label{eq:A}
A = \int_R |{\textrm d}\alpha| .
\end{equation}
In coordinates, $\alpha = \alpha_1 \d x + \alpha_2 \d y$, which gives $\d \alpha = (\partial_x\alpha_2 - \partial_y\alpha_1) \d x \wedge \d y$.
Defining $f = \partial_x\alpha_2 - \partial_y\alpha_1$, we have $\d \alpha = f \d x\wedge \d y$.
Now, applying the Cauchy-Schwarz inequality
\[ \left(\int_R f g \, \d x\wedge \d y\right)^2 \le \left(\int_R f^2\,\d x\wedge \d y\right)\left(
\int_R g^2\, \d x\wedge \d y\right)
\]
with $g\equiv 1$, we get
\begin{align*}
A& \le \left( \left( \int_R f^2\d x\wedge \d y\right) \hbox{area}(R)\right)^{\frac{1}{2}} \\
& \le \|\alpha\|_{H^1}\left( \hbox{area}(R)\right)^{\frac{1}{2}}\\
&\le \|\alpha\|_{H^1} (C \delta)^{\frac{1}{2}}
\end{align*}
where $C$ is a constant depending on $\phi$ (approximately equal to the length of $\phi$).
Therefore, for all such $\psi$ we have
\[\|[\phi] - [\psi]\|_{H^{-1}} = \sup_{\|\alpha\|_{H^1}=1} \left| [\phi](\alpha) - [\psi](\alpha)\right|\le
(C\delta)^{\frac{1}{2}},\]
establishing the claim.

For immersions that are not embeddings, Eq. \eqref{eq:A} is modified to take into 
account any intersections. Let  the region $R$ between $\phi$ and $\psi$ be
$R=\cup_{i=0}^n R_i$ where $R_0$ is the nonoverlapping part
and $R_1,\dots,R_n$ are the overlapping parts. Then
\[ A \le \sum_{i=0}^n d_i |f| \d x \wedge \d y \]
where $d_0=1$ and each $d_i$ is either 0 or 2, depending on the orientation of the
boundary curves of each $R_i$ (see Figure \ref{fig:proof}).
Applying the Cauchy--Schwarz inequality with $g=\sum_{i=0}^n d_i \mathcal{X}(R_i)$ gives
\[A \le \|\alpha\|_{H^1}\left( \sum_{i=0}^n d_i^2 \hbox{area}(R_i)\right)^{\frac{1}{2}}\]
As the number of intersections is fixed by the choice of $\phi$, again we have
$A\le (C\delta)^{\frac{1}{2}}$, establishing the claim.
\qed
 \end{proof}

\begin{figure}
\centering
\begin{tikzpicture}[scale=0.9]
	\begin{pgfonlayer}{nodelayer}
		\node [style=none] (0) at (-3, 1) {};
		\node [style=none] (1) at (-1, 3) {};
		\node [style=none] (2) at (1, 1) {};
		\node [style=none] (3) at (-1, -1) {};
		\node [style=none] (4) at (-4, 1) {};
		\node [style=none] (5) at (-1, 2) {};
		\node [style=none] (6) at (2, 1) {};
		\node [style=none] (7) at (-1, 0) {};
		\node [style=none] (8) at (-2.75, 2.5) {$\phi$};
		\node [style=none] (9) at (-3.75, 1.75) {$\psi$};
		\node [style=none] (10) at (-1, 2.5) {$D$};
		\node [style=none] (11) at (-3.5, 1) {$A$};
		\node [style=none] (12) at (-1, -0.5) {$B$};
		\node [style=none] (13) at (1.5, 1) {$C$};
	\end{pgfonlayer}
	\begin{pgfonlayer}{edgelayer}
		\draw [->, color=red, thick, bend left=45] (0.center) to (1.center);
		\draw [->, color=red, thick, bend left=45] (1.center) to (2.center);
		\draw [->, color=red, thick, bend left=45] (2.center) to (3.center);
		\draw [->, color=red, thick, bend left=45] (3.center) to (0.center);
		\draw [<-, color=blue, thick, in=180, out=93, looseness=0.75] (4.center) to (5.center);
		\draw [<-, color=blue, thick, in=87, out=0, looseness=0.75] (5.center) to (6.center);
		\draw [<-, color=blue, thick, in=0, out=258, looseness=0.75] (6.center) to (7.center);
		\draw [<-, color=blue, thick, in=267, out=180, looseness=0.75] (7.center) to (4.center);
	\end{pgfonlayer}
\end{tikzpicture}
\hfill
\begin{tikzpicture}[scale=0.9]
	\begin{pgfonlayer}{nodelayer}
		\node [style=none] (0) at (-5, 2) {};
		\node [style=none] (1) at (0, 2) {};
		\node [style=none] (2) at (-5, 1) {};
		\node [style=none] (3) at (0, 1) {};
		\node [style=none] (4) at (-3, 4) {};
		\node [style=none] (5) at (-2, 4) {};
		\node [style=none] (6) at (-3, -1) {};
		\node [style=none] (7) at (-2, -1) {};
		\node [style=none] (8) at (-3, 3.5) {};
		\node [style=none] (9) at (-3, 3) {};
		\node [style=none] (10) at (-2, 3.5) {};
		\node [style=none] (11) at (-2, 3) {};
		\node [style=none] (12) at (-3, 0.5) {};
		\node [style=none] (13) at (-3, -0.5) {};
		\node [style=none] (14) at (-2, 0.5) {};
		\node [style=none] (15) at (-2, -0.5) {};
		\node [style=none] (16) at (-0.5, 2) {};
		\node [style=none] (17) at (-1, 2) {};
		\node [style=none] (18) at (-0.5, 1) {};
		\node [style=none] (19) at (-1, 1) {};
		\node [style=none] (20) at (-4.5, 2) {};
		\node [style=none] (21) at (-4, 2) {};
		\node [style=none] (22) at (-4.5, 1) {};
		\node [style=none] (23) at (-4, 1) {};
		\node [style=none] (24) at (-2.1, 3.75) {};
		\node [style=none] (25) at (-2.1, 2.1) {};
		\node [style=none] (26) at (-2.9, 2.1) {};
		\node [style=none] (27) at (-2.9, 3.75) {};
		\node [style=none] (28) at (-0.25, 1.9) {};
		\node [style=none] (29) at (-1.9, 1.9) {};
		\node [style=none] (30) at (-1.9, 1.1) {};
		\node [style=none] (31) at (-0.25, 1.1) {};
		\node [style=none] (32) at (-2.9, -0.75) {};
		\node [style=none] (33) at (-2.9, 0.9) {};
		\node [style=none] (34) at (-2.1, 0.9) {};
		\node [style=none] (35) at (-2.1, -0.75) {};
		\node [style=none] (36) at (-4.75, 1.1) {};
		\node [style=none] (37) at (-3.1, 1.1) {};
		\node [style=none] (38) at (-3.1, 1.9) {};
		\node [style=none] (39) at (-4.75, 1.9) {};
		\node [style=none] (40) at (-4, 1.5) {$A$};
		\node [style=none] (41) at (-2.5, 0) {$B$};
		\node [style=none] (42) at (-1, 1.5) {$C$};
		\node [style=none] (43) at (-2.5, 3) {$D$};
		\node [style=none] (44) at (-0.5, 2.25) {$\phi(t)$};
		\node [style=none] (45) at (-0.5, 0.75) {$\psi(-t)$};
		\node [style=none] (46) at (-2.5, 1.5) {$E$};
	\end{pgfonlayer}
	\begin{pgfonlayer}{edgelayer}
		\draw [color=red, thick] (5.center) to (7.center);
		\draw [color=red, thick] (1.center) to (0.center);
		\draw [color=blue, thick] (6.center) to (4.center);
		\draw [color=blue, thick] (2.center) to (3.center);
		\draw [color=red, ->] (16.center) to (17.center);
		\draw [color=red, ->] (21.center) to (20.center);
		\draw [color=red, ->] (10.center) to (11.center);
		\draw [color=red, ->] (14.center) to (15.center);
		\draw [color=blue, ->] (13.center) to (12.center);
		\draw [color=blue, ->] (9.center) to (8.center);
		\draw [color=blue, ->] (22.center) to (23.center);
		\draw [color=blue, ->] (19.center) to (18.center);
		\draw [->] (36.center) to (37.center);
		\draw [->] (37.center) to (38.center);
		\draw [->] (38.center) to (39.center);
		\draw [->] (24.center) to (25.center);
		\draw [->] (25.center) to (26.center);
		\draw [->] (26.center) to (27.center);
		\draw [->] (28.center) to (29.center);
		\draw [->] (29.center) to (30.center);
		\draw [->] (30.center) to (31.center);
		\draw [->] (32.center) to (33.center);
		\draw [->] (33.center) to (34.center);
		\draw [->] (34.center) to (35.center);
	\end{pgfonlayer}
\end{tikzpicture}
\hfill
\begin{tikzpicture}[scale=0.9]
	\begin{pgfonlayer}{nodelayer}
		\node [style=none] (0) at (-5, 2) {};
		\node [style=none] (1) at (0, 2) {};
		\node [style=none] (2) at (-5, 1) {};
		\node [style=none] (3) at (0, 1) {};
		\node [style=none] (4) at (-3, 4) {};
		\node [style=none] (5) at (-2, 4) {};
		\node [style=none] (6) at (-3, -1) {};
		\node [style=none] (7) at (-2, -1) {};
		\node [style=none] (8) at (-3, 3.5) {};
		\node [style=none] (9) at (-3, 3) {};
		\node [style=none] (10) at (-2, 3.5) {};
		\node [style=none] (11) at (-2, 3) {};
		\node [style=none] (12) at (-3, 0.5) {};
		\node [style=none] (13) at (-3, -0.5) {};
		\node [style=none] (14) at (-2, 0.5) {};
		\node [style=none] (15) at (-2, -0.5) {};
		\node [style=none] (16) at (-0.5, 2) {};
		\node [style=none] (17) at (-1, 2) {};
		\node [style=none] (18) at (-0.5, 1) {};
		\node [style=none] (19) at (-1, 1) {};
		\node [style=none] (20) at (-4.5, 2) {};
		\node [style=none] (21) at (-4, 2) {};
		\node [style=none] (22) at (-4.5, 1) {};
		\node [style=none] (23) at (-4, 1) {};
		\node [style=none] (24) at (-2.1, 3.75) {};
		\node [style=none] (25) at (-2.1, 2.1) {};
		\node [style=none] (26) at (-2.9, 2.1) {};
		\node [style=none] (27) at (-2.9, 3.75) {};
		\node [style=none] (28) at (-0.25, 1.9) {};
		\node [style=none] (29) at (-1.9, 1.9) {};
		\node [style=none] (30) at (-1.9, 1.1) {};
		\node [style=none] (31) at (-0.25, 1.1) {};
		\node [style=none] (32) at (-2.9, -0.75) {};
		\node [style=none] (33) at (-2.9, 0.9) {};
		\node [style=none] (34) at (-2.1, 0.9) {};
		\node [style=none] (35) at (-2.1, -0.75) {};
		\node [style=none] (36) at (-4.75, 1.1) {};
		\node [style=none] (37) at (-3.1, 1.1) {};
		\node [style=none] (38) at (-3.1, 1.9) {};
		\node [style=none] (39) at (-4.75, 1.9) {};
		\node [style=none] (40) at (-4, 1.5) {$A$};
		\node [style=none] (41) at (-2.5, 0) {$B$};
		\node [style=none] (42) at (-1, 1.5) {$C$};
		\node [style=none] (43) at (-2.5, 3) {$D$};
		\node [style=none] (44) at (-0.5, 2.25) {$\phi(t)$};
		\node [style=none] (45) at (-0.5, 0.75) {$\psi(-t)$};
		\node [style=none] (46) at (-2.9, 1.9) {};
		\node [style=none] (47) at (-2.1, 1.9) {};
		\node [style=none] (48) at (-2.1, 1.1) {};
		\node [style=none] (49) at (-2.9, 1.1) {};
		\node [style=none] (50) at (-2.8, 1.8) {};
		\node [style=none] (51) at (-2.2, 1.8) {};
		\node [style=none] (52) at (-2.2, 1.2) {};
		\node [style=none] (53) at (-2.8, 1.2) {};
		\node [style=none] (54) at (-2.5, 1.5) {$E$};
	\end{pgfonlayer}
	\begin{pgfonlayer}{edgelayer}
		\draw [color=red, thick] (5.center) to (7.center);
		\draw [color=red, thick] (1.center) to (0.center);
		\draw [color=blue, thick] (6.center) to (4.center);
		\draw [color=blue, thick] (2.center) to (3.center);
		\draw [color=red, <-] (16.center) to (17.center);
		\draw [color=red, <-] (21.center) to (20.center);
		\draw [color=red, ->] (10.center) to (11.center);
		\draw [color=red, ->] (14.center) to (15.center);
		\draw [color=blue, ->] (13.center) to (12.center);
		\draw [color=blue, ->] (9.center) to (8.center);
		\draw [color=blue, <-] (22.center) to (23.center);
		\draw [color=blue, <-] (19.center) to (18.center);
		\draw [<-] (36.center) to (37.center);
		\draw [<-] (37.center) to (38.center);
		\draw [<-] (38.center) to (39.center);
		\draw [->] (24.center) to (25.center);
		\draw [->] (25.center) to (26.center);
		\draw [->] (26.center) to (27.center);
		\draw [<-] (28.center) to (29.center);
		\draw [<-] (29.center) to (30.center);
		\draw [<-] (30.center) to (31.center);
		\draw [->] (32.center) to (33.center);
		\draw [->] (33.center) to (34.center);
		\draw [->] (34.center) to (35.center);
		\draw [->] (46.center) to (47.center);
		\draw [->] (47.center) to (48.center);
		\draw [->] (48.center) to (49.center);
		\draw [->] (49.center) to (46.center);
		\draw [->] (50.center) to (51.center);
		\draw [->] (51.center) to (52.center);
		\draw [->] (52.center) to (53.center);
		\draw [->] (53.center) to (50.center);
	\end{pgfonlayer}
\end{tikzpicture}
\caption{\label{fig:proof}
Constructions used in the proof of Proposition~\ref{prop:cts}.
Here the red curve is the reference curve $\phi$ and the nearby (blue)
curve $\psi$ is shown with reversed orientation.
\emph{Left:} The line integral $\int_\phi\alpha+\int_\psi\alpha$ is equal, by Stokes' theorem,
to $\int_A\d\alpha-\int_B\d\alpha+\int_C\d\alpha-\int_D\d\alpha$, 
so its magnitude is bounded by $\int_{A\cup B\cup C\cup D}|\d\alpha|$.
\emph{Middle:} When the reference curve (in red) intersects itself in the sense shown,
$\int_{\phi}\alpha - \int_{\psi}\alpha = \int_A\d\alpha-\int_B\d\alpha+\int_C\d\alpha-\int_D\d\alpha$,
so $\int_E\d\alpha$ does not appear, i.e., it has weight $d=0$.
\emph{Right:} When the reference curve (in red) intersects itself in the other sense,
we have $\int_{\phi}\alpha - \int_{\psi}\alpha = -\int_A\d\alpha-\int_B\d\alpha-\int_C\d\alpha-\int_D\d\alpha
-2\int_E\d\alpha$, i.e., $\int_E\d\alpha$ appears with weight $d=2$.
A similar construction applies if the red curve takes on the same value 3 or more times.
}
\end{figure}
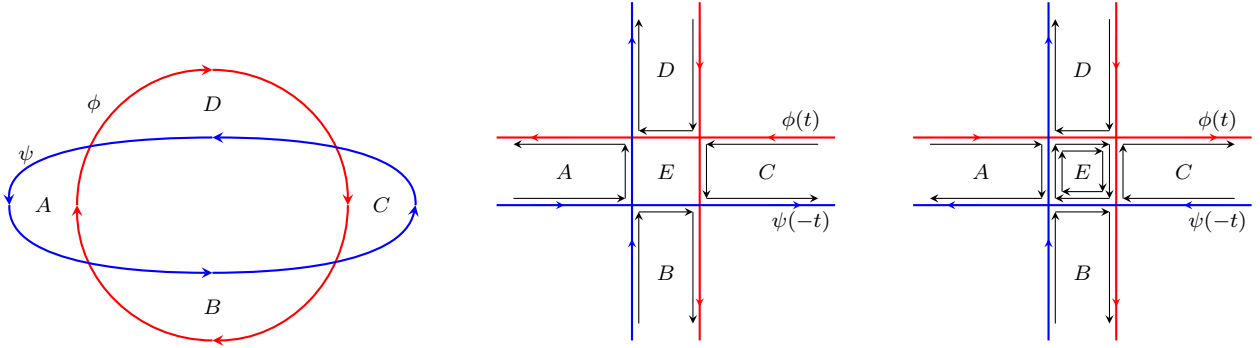

We now show that if the forms are smooth enough ($s \ge 2$), then the current map is differentiable.
In the case $s=1$, \autoref{prop:distanceperunit} indicates the map $[\cdot]$ is then not differentiable.

\begin{proposition}\label{prop:diff}
Let $s\ge 2$. Then the current map $[\cdot]\colon \LipImm(S^1,\Omega)\to H^{-s}\Lambda^1(\Omega)$ is differentiable.
\end{proposition}
\begin{proof}
Let $\xi(u,\cdot)$ be a $C^{1}$ curve in $\LipImm(S^1,\Omega)$ with $\xi(0)=\phi$
and $\partial_{u}|_{u=0}\xi(u)=X$. Let $\alpha\in H^s\Lambda^1(\Omega)$. 
Recall that for a 2-form $\omega$, the interior derivative $\mathrm{i}_X \omega$ is defined by $\mathrm{i}_X(Y) := \omega(X,Y)$.
We will show that the derivative of the current map at $\phi$ in direction $X$ is $\int_{\phi(S^1)} {\rm i}_X{\rm d}\alpha$.
By the trace theorem (see \autoref{rk:trace}), ${\rm d}\alpha\circ\phi\in H^{s-2}(S^1)$, and hence because $s\ge 2$, the derivative exists.

To establish this formula for the derivative, we first consider the case that the curves $\xi(u,\cdot)$ are embeddings.
Let $\varepsilon>0$.  Let $\tilde\xi(u,t) := (u,\xi(u,t))$ and $\pi(u,\tilde\xi)=\tilde\xi$.
Then $\tilde\xi([0,\varepsilon],S^1)$ is a tubular surface $M_\varepsilon$ with boundary $\Gamma_\varepsilon := \tilde\xi(\varepsilon,S^1) - \varepsilon(0,S^1)$. Then, using Stokes's theorem in the last line,
\begin{equation*}
\begin{aligned}
\Delta_\varepsilon &:= \int_{\xi(\varepsilon,S^1)}\alpha - \int_{\xi(0,S^1)}\alpha \\
&= \int_{\tilde\xi(\varepsilon,S^1)}\pi^*\alpha - \int_{\tilde \xi(0,S^1)}\pi^*\alpha \\
&= \int_{\Gamma_\varepsilon}\pi^*\alpha \\
&= \int_{M_\varepsilon} {\rm d}\pi^*\alpha .
\end{aligned}
\end{equation*}
Letting $\varepsilon\to 0$,
\begin{equation*}
\begin{aligned}
\frac{d}{d\varepsilon}[\xi(\varepsilon,\cdot)](\alpha)\big|_{\varepsilon=0} &= 
\lim_{\varepsilon\to0}\frac{\Delta_\varepsilon}{\varepsilon} \\
&= \int_{\tilde\xi(0,S^1)}{\rm i}_{\tilde\xi_*\frac{\partial}{\partial u}} {\rm d}\pi^*\alpha \\
&= \int_{\phi(S^1)}{\rm i}_X{\rm d}\alpha.
\end{aligned}
\end{equation*}
In the case that the curves $\xi(u,\cdot)$ are immersions, the range of the curves may be lifted
from $\R^2$ to $\R^3$ and the curves perturbed slightly at the crossings  so that they become
embeddings in $\R^3$. Then the same formula for the derivative holds in $\R^3$ and,
letting the perturbation tend to zero, in $\R^2$.

\qed
\end{proof}

Thus when $s\ge 2$ one can define a continuous \emph{Riemannian} metric on shapes as the restriction of 
$H^{-s}\Lambda^1(\Omega)$ to the currents of shapes. While many families of Riemannian metrics on shapes have been studied \cite{michor2007overview}, this one appears to be new. However, in this paper we do not use
the induced Riemannian metric, but rather the  (`straight line') subset metric illustrated
in Figure~\ref{fig:main}, which is far easier to compute. 

\subsection{Representers of shapes}
\label{sec:representers}
From \autoref{prop:bounded} the  
Riesz representation theorem applies: the current $[\phi]$ determines a unique $\beta\in H^s\Lambda^m(\Omega)$ 
called the (Riesz) \emph{representer} of $[\phi]$, which satisfies:
\begin{equation}
\label{eq:riesz}
[\phi](\alpha) = (\beta,\alpha)_{H^s}
\end{equation}
for all $\alpha\in H^s\Lambda^m(\Omega)$.
We can also write
 \begin{equation}
 \label{eq:norm}
  \| [\phi]_{-s} \|^{2} := \| \beta \|_{H^s}^2 = [\phi](\beta).
 \end{equation}
 That is, two shapes are close in $H^{-s}$ if the representers of their currents
 are close in  $H^s$.

From the definition of the representer, which is a PDE in weak form,
the representer satisfies an elliptic PDE with source concentrated on the shape.
We now take a closer look at the PDE in strong form in the case where both the representer and the shape are smooth.

Suppose now that $\phi(S^1)$ is smooth, and that the representer $\beta$ is smooth.
From Green's theorem, for $s=1$ we have
\begin{align}
\notag
(\alpha,\beta)_{H^1_{\scale}} &= \int_\Omega \Big(\sum_{i=1}^2 \alpha_i\beta_i + \scale^2 \sum_{i,j=1}^2 
\frac{\partial\alpha_i}{\partial x_j}\frac{\partial\beta_i}{\partial x_j}\Big)\, \ud x_1\, \ud x_2\\
&= \int_\Omega \Big(\sum_{i=1}^2 \alpha_i(1-\scale^2\nabla^2)\beta_i\Big) \, \ud x_1 \ud x_2
+ \int_{\partial \Omega}  \beta_{i n}\alpha_i \ud s
\label{eq:hh2}
\end{align}
where $\ud s$ is arclength of the boundary,
and $\beta_{i n}$ is the derivative of $\beta_i$ normal to the boundary of $\phi(S^1)$.

Now let $\delta_{\phi(S^1)}$ be the measure on $\Omega$ with support $\phi(S^1)$ that obeys
\begin{equation*}
  \langle  \delta_{\phi(S^1)}, f \rangle  = \int_0^{1} f\big(\phi(t)\big) \|\phi'(t)\|\,\ud t
  .
\end{equation*}
for all smooth $f\colon \Omega\to\R$.
Using this definition, we have
	\begin{equation}
		\begin{aligned}
			[\phi](\alpha) &= \int_{S^1} \phi^*\alpha \\
			&= \int_0^{1} \langle \alpha(\phi(t)),\phi'(t)\rangle\, \ud t\\
			&= \int_0^{1} \langle \alpha(\phi(t)),\mathbf{t}(t)\rangle \|\phi'(t)\|\, \ud t\\
			& = \langle  \delta_{\phi(S^1)} , \langle \alpha  , \mathbf{t} \rangle \rangle
      ,
		\end{aligned}
	\end{equation}
  where $\langle  \alpha , \mathbf{t} \rangle$ denotes (with a slight abuse of notation) any extension of that function outside the curve $\phi(S^1)$.
	Therefore
	\begin{equation*}
\int_\Omega \Big(\sum_{i=1}^2 \alpha_i(1-\scale^2\nabla^2)\beta_i\Big) \, \ud x_1 \ud x_2
+ \int_{\partial \Omega}  \beta_{i n}\alpha_i \ud s
= \langle  \delta_{\phi(S^1)} , \langle  \alpha, \mathbf{t} \rangle \rangle
\end{equation*}
for all $\alpha\in H^1\Lambda^1(\Omega)$.
This gives the Helmholtz equation
\begin{equation}
\label{eq:hh} (1-\scale^2 \nabla^2)\beta_i = \delta_{\phi(S^1)} \mathbf{t},\qquad \beta_{i n} = 0\text{ on } \partial\Omega
.
\end{equation}

		For a small scale $\scale$, the representer $\beta$ can be identified with a vector field concentrated near the curve and roughly
	tangent to it (an example is shown in Figure \ref{fig:rep}).
	The representer can be written in terms of the Green's function of the Helmholtz operator. 
	For example, when $\Omega=\R^2$, the Green's function of $(1-\scale^2\nabla^2)$ is
	\[
  G(x) = \frac{1}{2 \pi \scale^2}\mathsf{K}_0\left(\frac{\|x\|}{\scale}\right)
\]
  where $\mathsf{K}_0$ is the modified Bessel function of the 2nd kind; then
	\[
\beta(x) = \int_a^b G(x-\phi(t))\frac{\phi'(t)}{\|\phi'(t)\|}\, \ud t
.
\]


As remarked earlier, for immersions, the currents do not determine the shape, as parts of the shapes
that retrace themselves are invisible to currents.
For embeddings the situation is better:

\begin{proposition}
\label{prop:emb}
Let $\phi_1,\phi_2\colon S^1\to \Omega$ be two Lipschitz embeddings of $S^1$ and let $s= 1$. 
If $[\phi_1] = [\phi_2]$ then
the curves represent the same oriented shape.
\end{proposition}
\begin{proof}
Suppose that a point $x$ in $\phi_1(S^1)$ is not in $\phi_2(S^1)$.
Then
there is a neighbourhood of $x$ which does not intersect $\phi_2(S^1)$ either.
Choose a 1-form $\alpha$ that has support in this neighbourhood.
As $\phi_1(S^1)$ is an embedding, the form $\alpha$ can be chosen such that $[\phi_1](\alpha) \neq 0$.
However, $[\phi_2](\alpha) = 0$, which gives a contradiction.
\qed
\end{proof}

However, currents do determine the shape of immersions if there is some control over the self-intersections:

\begin{proposition}
\label{prop:selfint}
Let $\phi_1,\phi_2\colon S^1\to \Omega$ be two Lipschitz immersions, each with
a finite number of self-intersections such that at each self-intersection the tangent vectors are
continuous and distinct.
Let $s= 1$. 
If $[\phi_1] = [\phi_2]$ then
the curves represent the same oriented shape.
\end{proposition}
\begin{proof}
  As in the proof of \autoref{prop:emb}, the currents determine the images of $\phi_1$ and $\phi_2$ away from self-intersections.
  At the self-intersections, the hypothesis on
continuity and distinctness of the tangent vectors allows the non-self-intersecting pieces
of the shapes to be joined together in a unique way, thus determining the same oriented shape.
\qed
\end{proof}

The assumption on the self-intersections is necessary. Even without parts of curves
that retrace themselves, self-intersections---either with equal or discontinuous tangent vectors---%
prevent the current from recognising the shape up to $\Diff(S^1)$ reparameterizations 
(see Fig.~\ref{fig:selfintersect}). The current sees
only the image of the curve. 

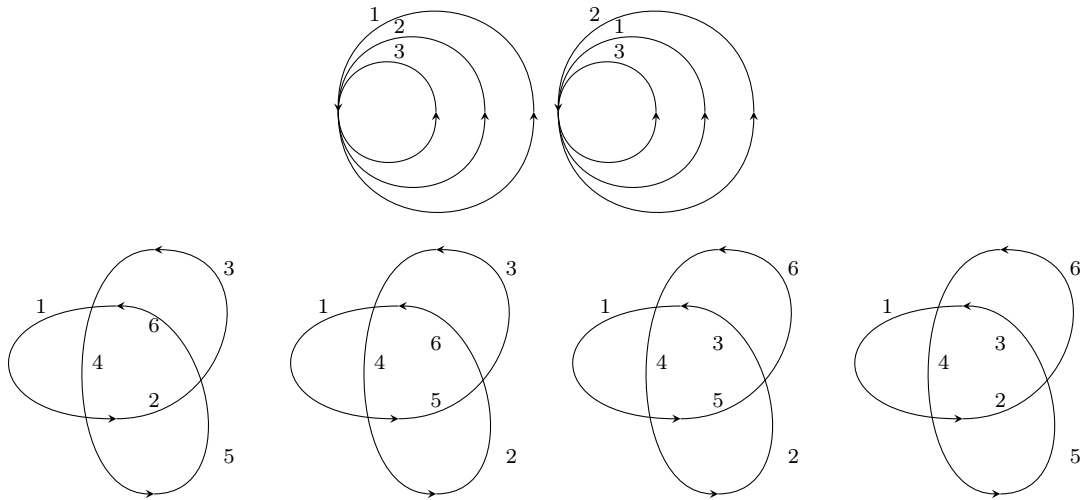
\begin{figure}
\centering
\begin{tikzpicture}[scale=0.65]
	\begin{pgfonlayer}{nodelayer}
		\node [style=none] (0) at (-5, 0) {};
		\node [style=none] (1) at (-1, 0) {};
		\node [style=none] (2) at (-2, 0) {};
		\node [style=none] (3) at (-3, 0) {};
		\node [style=none] (4) at (-4.25, 2) {1};
		\node [style=none] (5) at (-3.75, 1.75) {2};
		\node [style=none] (6) at (-3.75, 1.25) {3};
		\node [style=none] (7) at (0.75, 1.25) {3};
		\node [style=none] (8) at (3.5, 0) {};
		\node [style=none] (9) at (-0.5, 0) {};
		\node [style=none] (10) at (0.25, 2) {2};
		\node [style=none] (11) at (1.5, 0) {};
		\node [style=none] (12) at (2.5, 0) {};
		\node [style=none] (13) at (0.75, 1.75) {1};
	\end{pgfonlayer}
	\begin{pgfonlayer}{edgelayer}
		\draw [<-, bend left=90, looseness=1.75] (0.center) to (1.center);
		\draw [->, bend right=90, looseness=1.75] (0.center) to (3.center);
		\draw [->, bend right=90, looseness=1.75] (0.center) to (2.center);
		\draw [->, bend right=90, looseness=1.75] (0.center) to (1.center);
		\draw [bend left=90, looseness=1.75] (0.center) to (3.center);
		\draw [bend left=90, looseness=1.75] (0.center) to (2.center);
		\draw [<-, bend left=90, looseness=1.75] (9.center) to (8.center);
		\draw [->, bend right=90, looseness=1.75] (9.center) to (11.center);
		\draw [->, bend right=90, looseness=1.75] (9.center) to (12.center);
		\draw [->, bend right=90, looseness=1.75] (9.center) to (8.center);
		\draw [bend left=90, looseness=1.75] (9.center) to (11.center);
		\draw [bend left=90, looseness=1.75] (9.center) to (12.center);
	\end{pgfonlayer}
\end{tikzpicture}
\\
\begin{tikzpicture}
	\begin{pgfonlayer}{nodelayer}
		\node [style=none] (0) at (-6.75, 3.75) {};
		\node [style=none] (1) at (-7.25, 3) {};
		\node [style=none] (2) at (-7.25, 1.5) {};
		\node [style=none] (3) at (-6.75, 0.5) {};
		\node [style=none] (4) at (-8.25, 3) {1};
		\node [style=none] (5) at (-6.75, 1.75) {2};
		\node [style=none] (6) at (-5.75, 3.5) {3};
		\node [style=none] (7) at (-7.5, 2.25) {4};
		\node [style=none] (8) at (-5.75, 1) {5};
		\node [style=none] (9) at (-6.75, 2.75) {6};
		\node [style=none] (10) at (-3, 2.5) {6};
		\node [style=none] (11) at (-2, 3.5) {3};
		\node [style=none] (12) at (-2, 1) {2};
		\node [style=none] (13) at (-3.5, 1.5) {};
		\node [style=none] (14) at (-3.75, 2.25) {4};
		\node [style=none] (15) at (-3, 1.75) {5};
		\node [style=none] (16) at (-3, 0.5) {};
		\node [style=none] (17) at (-3, 3.75) {};
		\node [style=none] (18) at (-3.5, 3) {};
		\node [style=none] (19) at (-4.5, 3) {1};
		\node [style=none] (20) at (1.75, 1) {2};
		\node [style=none] (21) at (0.75, 3.75) {};
		\node [style=none] (22) at (1.75, 3.5) {6};
		\node [style=none] (23) at (0.75, 0.5) {};
		\node [style=none] (24) at (0, 2.25) {4};
		\node [style=none] (25) at (0.25, 1.5) {};
		\node [style=none] (26) at (0.25, 3) {};
		\node [style=none] (27) at (0.75, 2.5) {3};
		\node [style=none] (28) at (0.75, 1.75) {5};
		\node [style=none] (29) at (-0.75, 3) {1};
		\node [style=none] (30) at (5.5, 1) {5};
		\node [style=none] (31) at (4.5, 3.75) {};
		\node [style=none] (32) at (5.5, 3.5) {6};
		\node [style=none] (33) at (4.5, 0.5) {};
		\node [style=none] (34) at (3.75, 2.25) {4};
		\node [style=none] (35) at (4, 1.5) {};
		\node [style=none] (36) at (4, 3) {};
		\node [style=none] (37) at (4.5, 2.5) {3};
		\node [style=none] (38) at (4.5, 1.75) {2};
		\node [style=none] (39) at (3, 3) {1};
	\end{pgfonlayer}
	\begin{pgfonlayer}{edgelayer}
		\draw [->, bend right=90] (0.center) to (3.center);
		\draw [->, in=0, out=0, looseness=1.25] (3.center) to (1.center);
		\draw [->, bend right=90, looseness=3.25] (1.center) to (2.center);
		\draw [->, in=0, out=0, looseness=1.75] (2.center) to (0.center);
		\draw [->, bend right=90] (17.center) to (16.center);
		\draw [->, in=0, out=0, looseness=1.25] (16.center) to (18.center);
		\draw [->, bend right=90, looseness=3.25] (18.center) to (13.center);
		\draw [->, in=0, out=0, looseness=1.75] (13.center) to (17.center);
		\draw [->, bend right=90] (21.center) to (23.center);
		\draw [->, in=0, out=0, looseness=1.25] (23.center) to (26.center);
		\draw [->, bend right=90, looseness=3.25] (26.center) to (25.center);
		\draw [->, in=0, out=0, looseness=1.75] (25.center) to (21.center);
		\draw [->, bend right=90] (31.center) to (33.center);
		\draw [->, in=0, out=0, looseness=1.25] (33.center) to (36.center);
		\draw [->, bend right=90, looseness=3.25] (36.center) to (35.center);
		\draw [->, in=0, out=0, looseness=1.75] (35.center) to (31.center);
	\end{pgfonlayer}
\end{tikzpicture}
\caption{\label{fig:selfintersect}Currents see only the image of the curve and thus cannot distinguish some shapes.
In these shapes the numbers indicate the order in which the curves are traversed.
The shapes in each row have the same currents. In the top row, the curves are smooth,
but the curves intersect tangentially, and the current does not determine the
order of traversal.
In the bottom row, the current determines the shape
only if it is known that the curve is smooth at its self-intersections (bottom left); otherwise, there are four different orderings, that is, four distinct
elements of $\LipImm(S^1,\R^2)/\Diff^+(\R)$ have the same current.
}
\end{figure}

%
%

Insight into the $H^{-s}$ shape metric
is obtained by considering the target domain $\R^2$ and shapes consisting of vertical lines with periodic boundary conditions in $y$.  
\autoref{prop:distanceperunit}  illustrates both the non-differentiability of the current map $[\cdot]$ in the $s=1$ case and the difference between the $s=1$ and $s=2$ metrics. The $s=1$ metric weights nearby portions of the shapes more heavily than the $s=2$ metric does. It also shows the role of the length scale $\scale$; roughly, all curves
more than Euclidean distance $\scale$ apart are an equal distance apart in the $H^{-s}$ metrics.

We first prove an elementary Lemma.

\begin{lemma}
\label{prop:rkhs}
  Suppose that $H$ is a reproducing kernel Hilbert space of functions on $\R$.
  Let $\delta_x \in H^*$ denote the evaluation at the point $x\in \R$.
  Suppose further that the kernel is translation invariant, i.e., the representer for $\delta_x$ takes the form $x' \mapsto K(x' - x)$.
  Then the distance between $\delta_0$ and $\delta_{\epsilon}$ is
  \begin{equation}
    \| \delta_0 - \delta_{\epsilon} \| = \sqrt{2 (K(0) - K(\epsilon))}
    .
  \end{equation}
\end{lemma}
\begin{proof}
  In general, if $K_x$ denotes the representer of $\delta_x$, we have
  \begin{align*}
    \| \delta_0 - \delta_{\epsilon} \|^2 &= (K_0 - K_{\epsilon}, K_0 - K_{\epsilon}) \\
                                         &= \langle   \delta_0 - \delta_{\epsilon}, K_0 - K_{\epsilon} \rangle \\
                                         &= K_0(0) - K_{\epsilon}(0) - K_0(\epsilon) + K_{\epsilon}(\epsilon) \\
    &= K_0(0) + K_{\epsilon}(\epsilon) - 2K_0(\epsilon)
  \end{align*}
  Using the translation invariance of the kernel, we have $K_{\epsilon}(\epsilon) = K_0(0)$, which finishes the proof.
  \qed
\end{proof}

Assuming periodic boundary conditions in $y$, the representer of a vertical line at $x=x_0$ is $K((x-x_0)/\scale)$ where $K$ is the corresponding one-dimensional kernel.
We define the \emph{distance per unit length} to be the distance in one dimension, with the corresponding one-dimensional kernel.

\begin{proposition}
  \label{prop:distanceperunit}
The distance per unit length between two straight lines a (Euclidean) distance $\epsilon$ 
apart is $\mathcal{O}(\epsilon^{1/2})$  when $s=1$ and $\mathcal{O}(\epsilon)$ when $s=2$.
\end{proposition}
\begin{proof}
Using \autoref{prop:rkhs}, the distance by unit length is thus
\(
   \sqrt{2(K(0) - K(\epsilon))}
\).

For $s=1$, $K(x)=\frac{1}{2}\mathrm{e}^{-|x|}$ and the distance per unit length between the lines is
$(1-\mathrm{e}^{-\epsilon/\scale})^{1/2} = \mathcal{O}(\epsilon^{1/2})$
as $\epsilon\to 0$.

For $s=2$, $K(x) = \frac{1}{4}\mathrm{e}^{-|x|}(1 + |x|)$ and the distance per unit length between the lines is
$(\frac{1}{2}(1-\mathrm{e}^{-(\epsilon/\scale)}(1+\epsilon/\scale)))^{1/2} = \mathcal{O}(\epsilon)$
as $\epsilon\to 0$.
\qed
\end{proof}

We now turn to a practical implementation of shape representation using currents that is based on the finite element method. 

\section{Discretization of currents by finite elements}
\label{sec:finel}

The discretization of the representation of shapes by currents proceeds in two steps:
\begin{itemize}
\item[(i)] Discretization of the space of currents and approximation of the currents of shapes; and
\item[(ii)] Discretization of the metric on currents.
\end{itemize}
We shall consider these separately, as their numerical properties are somewhat independent; item (i) 
determines how accurately the shapes themselves are represented, while item (ii) determines the
geometry of the induced discrete shape space.

\subsection{Discretization of currents}
We first return to the general setting in which we work with  immersions $\phi\colon M\to N$ where $M$ and $N$ are manifolds
and with their currents $[\phi](\alpha) = \int_M\phi^*\alpha$ for $\alpha\in H^s\Lambda^m(N)$.
The discretization of currents requires the choice of three things:
\begin{itemize}
\item[(i)] A space $V$ of finite elements on $M$;
\item[(ii)] A space $W$ of finite elements on $\Lambda^m(N)$;
\item[(iii)] A method of evaluating or approximating $[\phi_V]|_W$, where $\phi_V$ is
the finite element representation of $\phi$ in $V$.
\end{itemize}
We will explore analytically and numerically the ability of $[\phi_V]|_W$ to represent shapes in different settings.

To illustrate the ideas, we first consider an extremely simple example, namely sets of $n$ points on a line.

\begin{example}
Let $M=\{1,\dots,n\}$ and $N=\R$. A shape is then an unordered set of $n$ points in~$\R$ with isotropy the group of permutations.
In this case $M$ does not need to be discretized.
We choose a uniform mesh with spacing $\Delta x$ on $\R$ and let $W$ be the piecewise polynomials of degree at most $d$ on the mesh.
Let $w(x)\ud x\in W$. Then 
$$[\phi](w(x)\ud x) = \int_M \phi^*(w(x)\ud x) = \sum_{j=1}^n w(x_j).$$

The simplest case is $d=0$, i.e., piecewise constant elements. 
For these, the currents count how many points are in each cell. Therefore these elements represent
the shape with an accuracy of $\O(\Delta x)$, as there is no way to tell where in each cell 
the points are located.

The next case is $d=1$, i.e., piecewise linear elements. 
The piecewise constants determine how many points are in each cell, 
and the linear elements determine $\sum_{j\colon x_j\in W_i}x_j$, that is, the mean location of the points in cell $W_i$.
If there is at most 1 point in each cell, then the representation is perfect, as the points are located exactly. If
there is more than 1 point in a cell, then these elements represent the shape with an 
accuracy of  $\O(\Delta x)$. Notice how the use of currents factors out the parameterization
of the shape, i.e., it is invariant under  permutations of $M$.

With piecewise elements of degree $d$, the number of points in each cell and their first $d$ moments
are determined; if there are at most $d$ points in each cell, the points are located exactly. This is because
the moment equations on $[0,\Delta x]$ (for example) are $\sum_{j\colon x_j\in W_i} (x_j)^i = c_i$. For given values of the moments
$c_i$, this set of polynomial equations has total degree $d!$, hence at most $d!$ solutions; but if there
is one real solution, then any permutation of the $x_j$ is a solution; hence the moments determine
the points up to ordering.
\end{example}

We now consider our main example of oriented closed planar curves. Let $M=S^1$ and let $N=\Omega$, a domain in $\R^2$. We will choose $V$ to be the continuous piecewise polynomials of a given degree on a uniform mesh on $S^1$
and $W$ to be either the discontinuous or the continuous piecewise polynomials of a given degree on a fixed mesh on $\Omega$. We will take the finite element representation $\phi_V$ of $\phi$ to be the element of $V$ that interpolates $\phi$ at the finite element nodes. The current $[\phi_V]|_W$ is then given by the integral of a piecewise polynomial function. This can be evaluated exactly; however, in this study we evaluate $[\phi_V]_W$ by quadrature, either the midpoint rule for
piecewise linear elements or Simpson's rule for piecewise quadratics. As we are interested in fairly compact 
representations of shapes, we will pick the meshsize of $V$ to be much smaller than the 
meshsize of $W$.

In the following subsections we study the behaviour of finite element currents with regard to
(i) quadrature errors, with and without noise; (ii) accuracy of representation of shapes; and (iii) accuracy
of the induced metric on shapes.

\subsection{Quadrature errors and robustness of currents.}
\label{sec:quadrature}
One strong motivation for considering currents, as opposed to other possible
shape invariants such as those based on arclength (cf. \autoref{prop:current}), is that---%
essentially because they are signed---%
currents are expected to be robust against noise and to function well on quite rough shapes. 
We present three examples
measuring different aspects of robustness.

\begin{example} In this example the shapes are rough, but there is no noise.
  \label{ex:errornonsmooth}
\autoref{fig:errornonsmooth} shows the quadrature errors for 3 rough shapes
as a function of the number of points used to discretize them. In this and the following example, 
the $H^{-1}$ shape norm is discretized using polynomials of degree 9 on $[-1,1]^2$, although only
the quadrature error is reported. Shapes (a) and (b) are continuous, but not Lipschitz continuous,
and yet the quadrature errors are well under control, and for shape (b) are even $\O((\Delta s)^2)$ (where $\Delta s$ is the mean mesh spacing).
\end{example}

\begin{example}
\label{ex:error}
In this example, shown in \autoref{fig:error}, the shape is smooth, but different levels of noise are added. Thus, the exact value
to which we compare is the zero noise, zero mesh spacing limit. Two different quadratures
(the midpoint and Simpson's rule) are compared. Although high levels of noise can dominate the quadrature
error, they do not prevent the computation of highly accurate values for the currents until
the noise level is actually greater than the mesh spacing $\Delta s$.
\end{example}

\begin{example}
  \label{ex:errorsimple}
We compare the sensitivity of currents to that of arclength-based currents in \autoref{fig:errorsimple}.
The reference shape is a unit line segment, discretized with uniform mesh spacing $\Delta s$, and
independent normally-distributed noise of mean 0 and standard deviation $\epsilon$ is
added to the $x$ and $y$ components of each point except the endpoints. The current $\int y\, {\textrm d}x$ is computed by the midpoint approximation $\sum(y_{i+1}-y_i)(x_{i+1}-x_i)$.
For $\epsilon\ll \Delta s$,
the errors accumulate like a sum of random variables, the error being $N(0,\epsilon(\Delta s)^{\frac{1}{2}})$. 
For $\epsilon\gtrsim\Delta s$, nonlinear effects create a larger error of size $\O(\epsilon(\Delta s)^{-\frac{1}{2}})$.
Overall, good results are obtained even with noise levels $\epsilon = \O(\Delta s)$.
In contrast, the positive nature of length means that the arclength approximation
$\sum((x_{i+1}-x_i)^2 + (y_{i+1}-y_i)^2)^{\frac{1}{2}}$ accumulates errors of size
$\O((1+(\epsilon/\Delta s)^2)^{\frac{1}{2}}-1)$.
These are $\O(\epsilon^2/(\Delta s)^2)$ for small $\epsilon$ and $\O(\epsilon/\Delta s)$ for larger $\epsilon$. Overall, $\epsilon = o(\Delta s)$ is necessary for acceptable results.
\end{example}

\begin{figure}
\begin{minipage}[b]{0.4\textwidth}
\includegraphics[width=\textwidth]{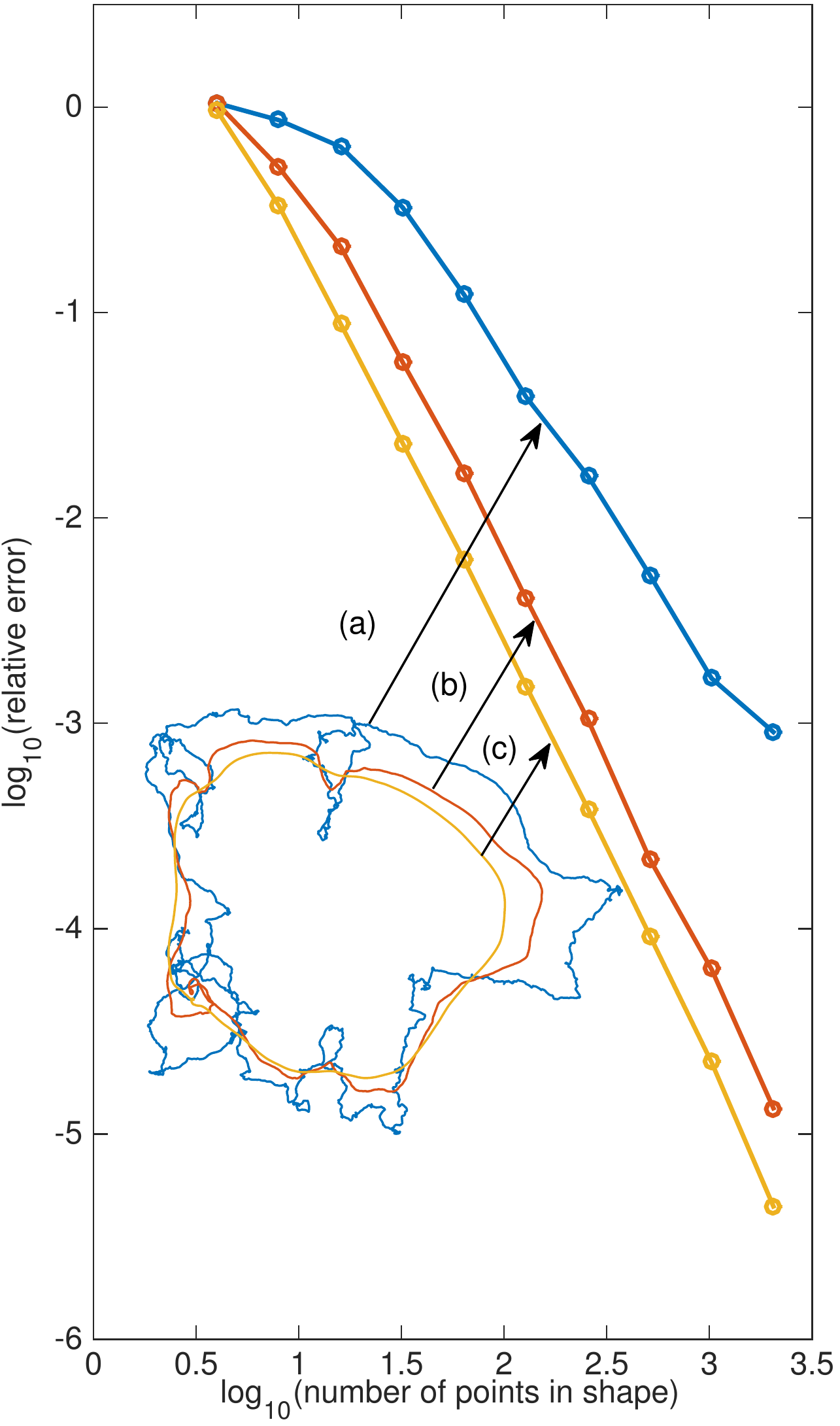}
\end{minipage}\hfill
\begin{minipage}[b]{0.55\textwidth}
\caption{\label{fig:errornonsmooth}
  \autoref{ex:errornonsmooth}:
Discretized currents can describe  rough shapes robustly. Here, 3 shapes of different smoothness are captured at different resolutions and the quadrature error in their norms (as computed with the midpoint rule) compared. Shape (a) has Fourier coefficients  $\tilde\phi_k=\O(|k|^{-3/2})$; it is in $H^p$ for $p<1$ and is continuous, but not differentiable. Shape (b) has Fourier coefficients $\tilde\phi_k=\O(|k|^{-2})$; it is in $H^p$ for $p<3/2$ and is continuous, but just fails to be differentiable. Shape (c) has Fourier coefficients
$\tilde\phi_k=\O(|k|^{-3})$; it is nearly twice differentiable. Even the very rough shape (a) can be represented with an error
of 1\% using 256 points.
}
\end{minipage}
\end{figure}

\begin{figure}
\begin{minipage}[b]{0.4\textwidth}
\includegraphics[width=\textwidth]{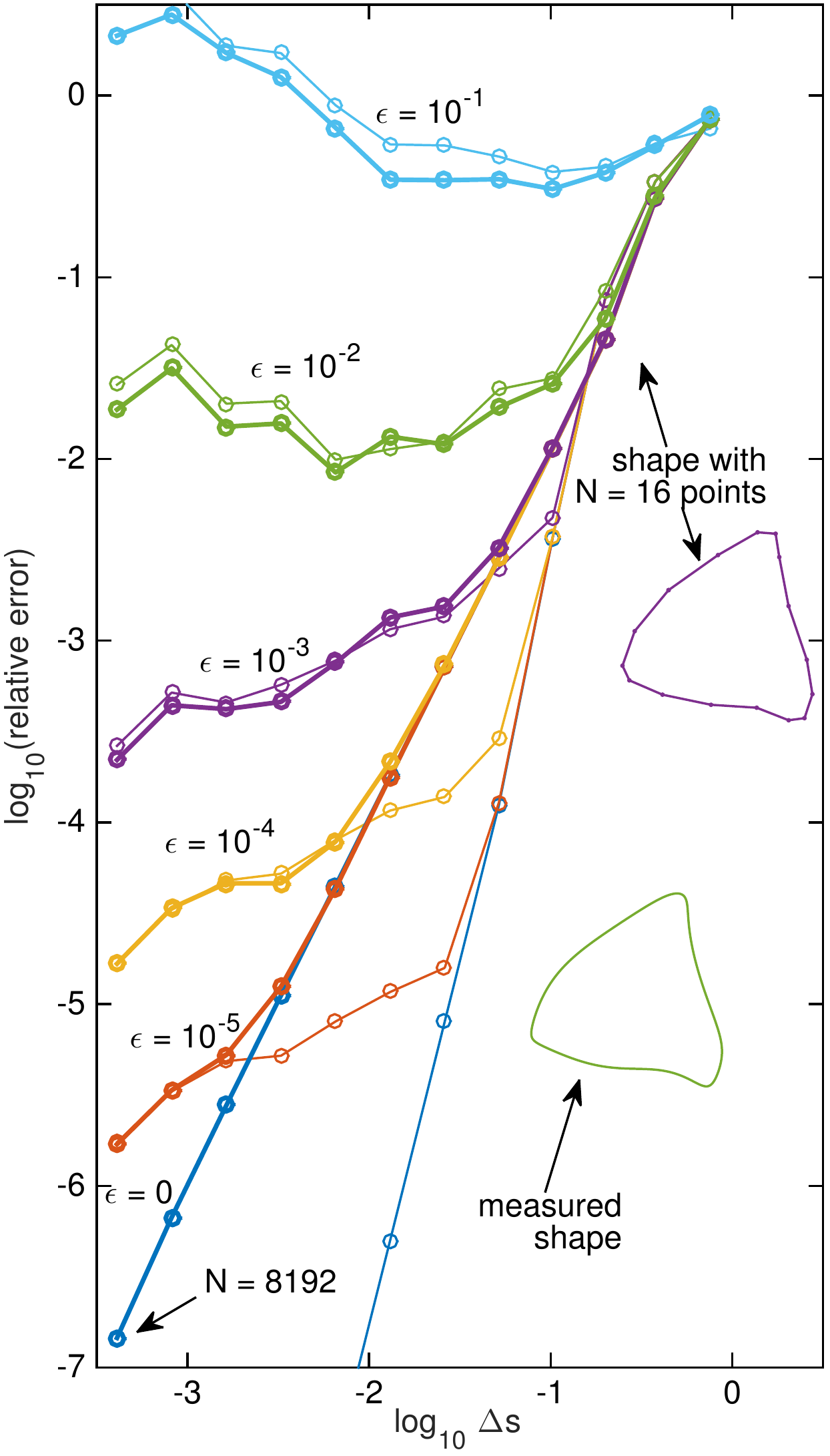}
\end{minipage}\hfill
\begin{minipage}[b]{0.55\textwidth}
\caption{\label{fig:error}
\autoref{ex:error}:
Errors in the shape norm $\|\phi\|^*$ calculated for different resolutions and noise levels.
The shape to be measured  is smooth and has 13 random nonzero Fourier coefficients. The relative
error in its shape is shown as calculated by the midpoint rule (thick lines) and
by Simpson's rule (thin lines), as a function of the mean mesh spacing $\Delta s := \|\phi(t_{i+1})-\phi(t_i)\|$. 
Independent normally distributed noise with standard
deviation $\epsilon$ is added to each point. When $\epsilon=0$, the 
 error is $\O((\Delta s)^2)$  for the midpoint rule and $\O((\Delta s)^4)$ for Simpson's rule, as expected.
The shape can be represented with an error of 5\%
with 16 points and a noise of 0.05 (5\%), or with an error of 0.1\% with 256 points and noise
of 0.1\%.
}
\end{minipage}
\end{figure}

\begin{figure}
\begin{minipage}[b]{0.4\textwidth}
\includegraphics[width=\textwidth]{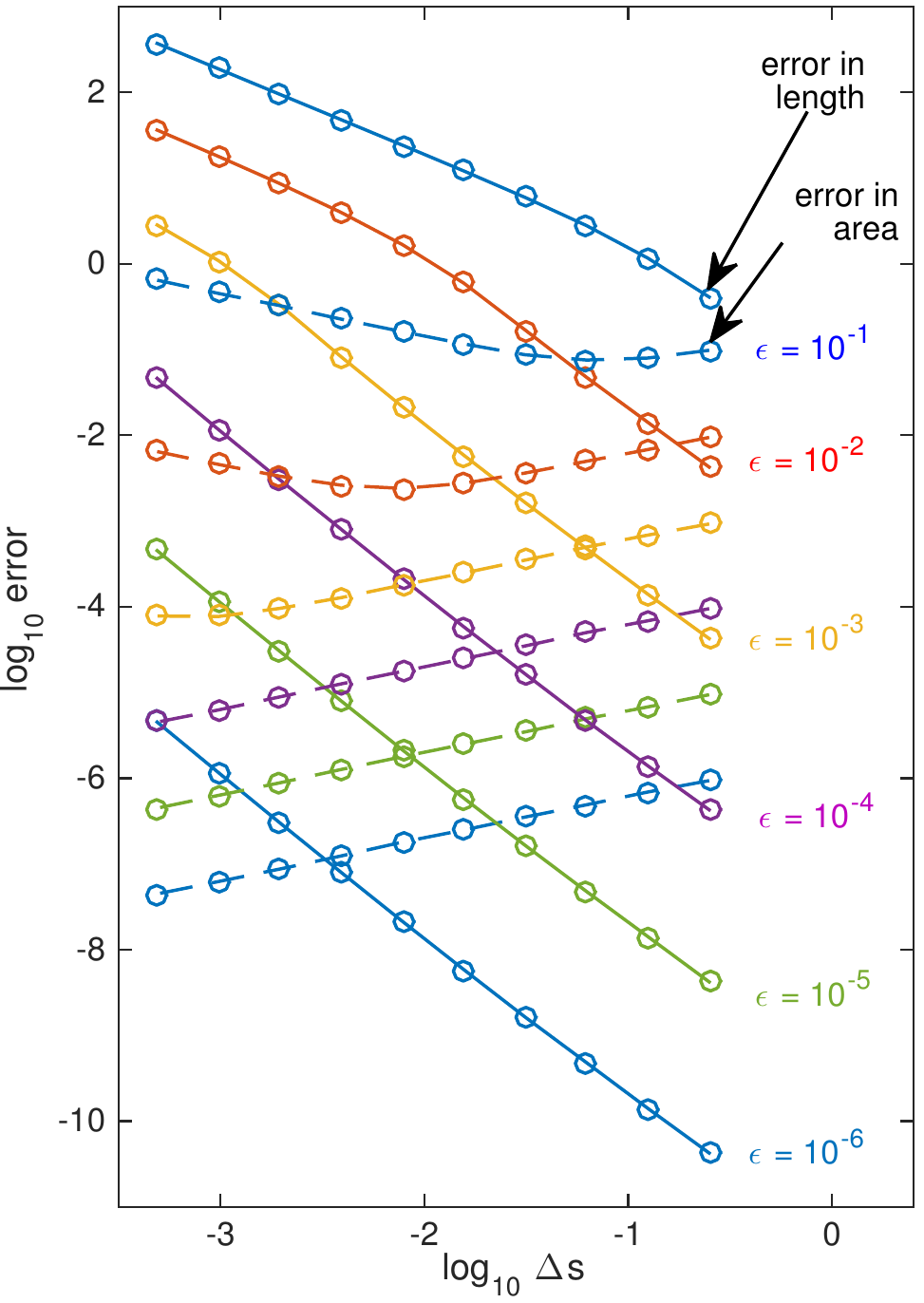}
\end{minipage}\hfill
\begin{minipage}[b]{0.55\textwidth}
\caption{\label{fig:errorsimple}
\autoref{ex:errorsimple}:
A comparison of the robustness of currents vs. arclength-based invariants.
Here the reference shape is the unit line segment $([0,1],0)$ discretized
with mesh spacing $\Delta s$. The current $\int y\, \d x $, the signed
area under the curve, is
calculated using the midpoint rule, and the arclength is calculated from
a piecewise linear approximation. The data $x_i$ and $y_i$, except
the endpoints, are perturbed with many repeated trials of
independent Gaussian noise of mean 0 and standard deviation $\epsilon$.
The 2-standard-deviation limits of the error in the results are shown.
For the current, the errors (dashed lines) are $\O(\epsilon (\Delta s)^{1/2} + 
\epsilon (\Delta s)^{-1/2})$, so that excellent results are obtained
even with $\epsilon\sim\Delta s$; for the arclength (solid lines), the errors
are $\O(\epsilon/\Delta s)^2)$ for small $\epsilon$, so that $\epsilon=o(\Delta s)^2$
is required for acceptable results.
}
\end{minipage}
\end{figure}

\subsection{Accuracy of representation of shapes}
\label{sec:approx}

In this section we consider simple closed planar shapes and study how accurately they are represented by discontinuous finite elements. Let $M=S^1$ and let $N=\Omega$, a domain in $\R^2$.  Let $\phi\colon S^1\to\Omega$ be a smooth embedding. Take a triangular mesh on $\Omega$
consisting of triangles $\mathcal{T}$ of maximum diameter $\Delta x$.
Let $W$ be the space of discontinuous piecewise polynomial finite elements of degree $\le d$. We are studying the effect of the discretization of the ambient space $N$, so we do not discretize $M$; we assume that all currents are evaluated exactly. We assume that the currents $[\phi]|_W$ are given and we
 want to know how accurately they determine the shape.

\begin{figure}
\centering
\begin{tikzpicture}
	\begin{pgfonlayer}{nodelayer}
		\node [style=none] (0) at (-1.25, 1.5) {};
		\node [fill=green, circle, style=none, minimum size=1.5 mm] (1) at (-1.65, 0.1) {};
		\node [fill=blue, circle, style=none, minimum size=1.5 mm] (2) at (-2.1, 1.7) {};
		\node [style=none] (3) at (-3.5, 2) {};
		\node [fill=blue, circle, style=none, minimum size=1.5 mm] (4) at (-2.1, 0.2) {};
		\node [style=none] (5) at (-1.75, -0.25) {};
		\node [fill=green, circle, style=none, minimum size=1.5 mm] (6) at (-1.65, 1.6) {};
	\end{pgfonlayer}
	\begin{pgfonlayer}{edgelayer}
		\draw (3.center) to (0.center);
		\draw (5.center) to (0.center);
		\draw (3.center) to (5.center);
		\draw [color=blue] (4.center) to (2.center);
		\draw [color=green] (6.center) to (1.center);
	\end{pgfonlayer}
\end{tikzpicture}
\caption{\label{fig:2segments}Any nonzero vector can be placed in at most two positions in any triangle such that the vector's head and tail lie on the sides of the triangle. 
If there are two such positions, then they are incident on different pairs of sides.
There is one such position when the vector is parallel to one of the sides.
}
\end{figure}
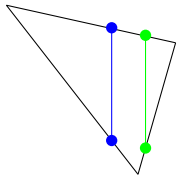

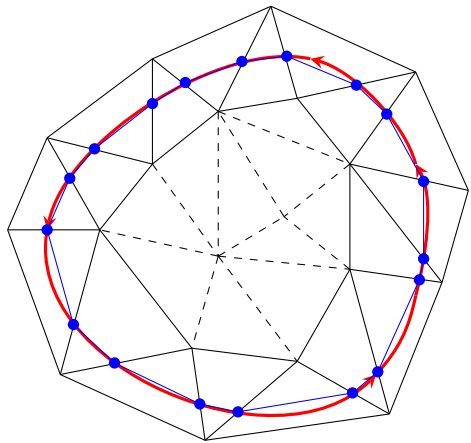
\begin{figure}
\centering
\begin{tikzpicture}[scale=0.7]
	\begin{pgfonlayer}{nodelayer}
		\node [fill=blue, circle, style=none, minimum size=1.5 mm] (0) at (-5, 0) {};
		\node [style=none] (1) at (0, 3.25) {};
		\node [style=none] (2) at (2, 1.25) {};
		\node [style=none] (3) at (2, -1.25) {};
		\node [style=none] (4) at (1.25, -2.75) {};
		\node [style=none] (5) at (-5.75, 0) {};
		\node [style=none] (6) at (-4, 0) {};
		\node [style=none] (7) at (-5, 1.75) {};
		\node [style=none] (8) at (-3, 1.25) {};
		\node [style=none] (9) at (-3, 3.25) {};
		\node [style=none] (10) at (-1.75, 2.25) {};
		\node [style=none] (11) at (-0.75, 4.25) {};
		\node [style=none] (12) at (-0.25, 2.5) {};
		\node [style=none] (13) at (2, 3) {};
		\node [style=none] (14) at (0.75, 1.25) {};
		\node [style=none] (15) at (3, 0.75) {};
		\node [style=none] (16) at (0.75, -0.75) {};
		\node [style=none] (17) at (-0.25, -2.5) {};
		\node [style=none] (18) at (-2, -4) {};
		\node [style=none] (19) at (-2.25, -2.25) {};
		\node [style=none] (20) at (-4.75, -2.75) {};
		\node [style=none] (21) at (-0.5, 0.25) {};
		\node [style=none] (22) at (-1.75, -0.5) {};
		\node [style=none] (23) at (2.5, -1) {};
		\node [style=none] (24) at (1.5, -3.5) {};
		\node [fill=blue, circle, style=none, minimum size=1.5 mm] (25) at (2.15, 0.92) {};
		\node [fill=blue, circle, style=none, minimum size=1.5 mm] (26) at (2.15, -0.55) {};
		\node [fill=blue, circle, style=none, minimum size=1.5 mm] (27) at (2.07, -0.95) {};
		\node [style=none] (28) at (1.25, -2.75) {};
		\node [fill=blue, circle, style=none, minimum size=1.5 mm] (29) at (1.28, -2.7) {};
		\node [fill=blue, circle, style=none, minimum size=1.5 mm] (30) at (0.8, -3.1) {};
		\node [fill=blue, circle, style=none, minimum size=1.5 mm] (31) at (-1.375, -3.46) {};
		\node [fill=blue, circle, style=none, minimum size=1.5 mm] (32) at (-2.1, -3.31) {};
		\node [fill=blue, circle, style=none, minimum size=1.5 mm] (33) at (-3.72, -2.53) {};
		\node [fill=blue, circle, style=none, minimum size=1.5 mm] (34) at (-4.5, -1.8) {};
		\node [fill=blue, circle, style=none, minimum size=1.5 mm] (35) at (-4.57, 0.98) {};
		\node [fill=blue, circle, style=none, minimum size=1.5 mm] (36) at (-4.1, 1.54) {};
		\node [fill=blue, circle, style=none, minimum size=1.5 mm] (37) at (-3, 2.4) {};
		\node [fill=blue, circle, style=none, minimum size=1.5 mm] (38) at (-2.375, 2.8) {};
		\node [fill=blue, circle, style=none, minimum size=1.5 mm] (39) at (-1.3, 3.2) {};
		\node [fill=blue, circle, style=none, minimum size=1.5 mm] (40) at (-0.45, 3.3) {};
		\node [fill=blue, circle, style=none, minimum size=1.5 mm] (41) at (0.8725, 2.75) {};
		\node [fill=blue, circle, style=none, minimum size=1.5 mm] (42) at (1.45, 2.2) {};
	\end{pgfonlayer}
	\begin{pgfonlayer}{edgelayer}
		\draw [very thick, color=red, <-, bend left, looseness=0.75] (1.center) to (2.center);
		\draw [very thick, color=red, <-, in=75, out=-60, looseness=0.75] (2.center) to (3.center);
		\draw [very thick, color=red, bend left=15] (3.center) to (4.center);
		\draw [very thick, color=red, <-, in=260, out=225] (4.center) to (0.center);
		\draw [very thick, color=red, <-, bend left=45, looseness=0.75] (0.center) to (1.center);
		\draw (5.center) to (7.center);
		\draw (7.center) to (6.center);
		\draw (6.center) to (5.center);
		\draw (7.center) to (8.center);
		\draw (8.center) to (6.center);
		\draw (9.center) to (8.center);
		\draw (9.center) to (7.center);
		\draw (9.center) to (10.center);
		\draw (10.center) to (8.center);
		\draw (10.center) to (11.center);
		\draw (11.center) to (9.center);
		\draw (10.center) to (12.center);
		\draw (12.center) to (11.center);
		\draw (12.center) to (13.center);
		\draw (13.center) to (11.center);
		\draw (14.center) to (13.center);
		\draw (12.center) to (14.center);
		\draw (14.center) to (15.center);
		\draw (15.center) to (13.center);
		\draw (23.center) to (15.center);
		\draw (14.center) to (23.center);
		\draw (16.center) to (23.center);
		\draw (16.center) to (14.center);
		\draw (16.center) to (24.center);
		\draw (24.center) to (23.center);
		\draw (16.center) to (17.center);
		\draw (17.center) to (24.center);
		\draw (24.center) to (18.center);
		\draw (18.center) to (17.center);
		\draw (19.center) to (18.center);
		\draw (19.center) to (17.center);
		\draw (19.center) to (20.center);
		\draw (20.center) to (18.center);
		\draw (6.center) to (19.center);
		\draw (6.center) to (20.center);
		\draw (20.center) to (5.center);
		\draw [dashed] (10.center) to (14.center);
		\draw [dashed] (14.center) to (21.center);
		\draw [dashed] (21.center) to (10.center);
		\draw [dashed] (22.center) to (21.center);
		\draw [dashed] (8.center) to (22.center);
		\draw [dashed] (22.center) to (6.center);
		\draw [dashed] (22.center) to (19.center);
		\draw [dashed] (22.center) to (17.center);
		\draw [dashed] (22.center) to (16.center);
		\draw [dashed] (16.center) to (21.center);
		\draw [dashed] (10.center) to (22.center);
		\draw [color=blue] (41.center) to (40.center);
		\draw [color=blue] (40.center) to (39.center);
		\draw [color=blue] (39.center) to (38.center);
		\draw [color=blue] (38.center) to (37.center);
		\draw [color=blue] (37.center) to (36.center);
		\draw [color=blue] (36.center) to (35.center);
		\draw [color=blue] (35.center) to (0.center);
		\draw [color=blue] (0.center) to (34.center);
		\draw [color=blue] (34.center) to (33.center);
		\draw [color=blue] (33.center) to (32.center);
		\draw [color=blue] (32.center) to (31.center);
		\draw [color=blue] (31.center) to (30.center);
		\draw [color=blue] (30.center) to (29.center);
		\draw [color=blue] (29.center) to (27.center);
		\draw [color=blue] (27.center) to (26.center);
		\draw [color=blue] (26.center) to (25.center);
		\draw [color=blue] (25.center) to (42.center);
		\draw [color=blue] (42.center) to (41.center);
	\end{pgfonlayer}
\end{tikzpicture}
\caption{\label{fig:pc} Approximation of a shape by piecewise constant currents (built up from Figure~\ref{fig:easy}). A smooth curve
is shown in red and a triangular finite element mesh in black. The unique approximating curve that is continuous, 
is linear on each triangle, and interpolates the shape at the element edges is shown in blue.
On each  triangle, the currents evaluated on the constant 1-forms ${\textrm d}x$ and ${\textrm d}y$ determine the jumps in $x$ and $y$ of the shape across the triangle. These jumps determine (i) if the shape intersects the triangle; and (ii) at most two line segments with the given jumps. From these two, 
the line segment that meets other active edges is selected, yielding the second-order accurate approximant shown.}
\end{figure}

First we consider  piecewise constant elements, that is, $d=0$.

\begin{proposition}
\label{prop:constants}
For smooth embeddings $\phi\colon S^1\to\Omega\subset\R^2$, and a triangular mesh
$\mathcal{T}$ of sufficiently small mesh size $\Delta x$, the currents of
$\phi$ evaluated on 1-forms constant on each triangle $T\in{\mathcal T}$ determine
a piecewise linear approximation $\hat\phi$ of $\phi$ of pointwise second order accuracy.
\end{proposition}

\begin{proof}
The currents are 
$$\left\{\left(\int_{\phi(S^1)\cap T}\, \d x, \int_{\phi(S^1)\cap T}\, \d y\right)\colon T\in\mathcal{T}\right\}.$$
If the mesh is sufficiently fine, then these currents are either zero (if the curve does not intersect $T$), or they record the jumps in $x$ and $y$ of that part of the curve that lies in $T$. 
The elements on which the currents are nonzero therefore determine the set of elements whose interiors  
intersect the curve $\phi$. If the mesh is sufficiently fine and the curve is in general position, then these
elements form a `discrete topological circle', a set of triangles each sharing an edge with exactly two others
(see Figure \ref{fig:pc}).

We now consider finding a continuous shape $\hat\phi$ with the same currents as $\phi$. The
currents have two degrees of freedom per triangle, as do shapes that are linear on each
triangle; we therefor seek such an approximant $\hat\phi$. 
For any values of the currents of $\phi$ on piecewise constants, i.e., for any values of the jumps in $x$ and $y$, there are at most two line segments in $T$ with endpoints on the edges of $T$ whose currents take on these values. (See Figure \ref{fig:2segments}.) If there are two such line segments, then they join different pairs of edges. The line segment that joins two edges that are part of the known discrete topological circle can then be chosen. See Figure \ref{fig:pc}.

This piecewise linear approximation $\hat\phi$ to $\phi$ interpolates $\phi$ at the edges of the elements, and, as $\phi$ is assumed to be smooth, obeys
$$ \max_{s\in{S^1}}\, \min_{t\in S^1} \|\phi(t)-\hat\phi(s)\| = \mathcal{O}((\Delta x)^2)$$
on each triangle $T$. It therefore determines the shape to second order accuracy.
\qed
\end{proof}

Next we consider the improvement that can be obtained using
discontinuous piecewise linear or quadratic elements, i.e., $d=1$, 2.

\begin{proposition}
\label{prop:approx}
Let $\mathcal{T}$ be a triangular planar mesh. Let 
\begin{equation*}
\begin{aligned}
V_x &:=\{\d x,\d y,y \, \d x, x y \, \d x, y^2\, \d x\},\\
V_y&:=\{\d x, \d y, x\, \d y, x^2\, \d y, x y\, \d y\}.
\end{aligned}
\end{equation*}
 Then for  sufficiently
smooth $\phi\colon S^1\to \R^2$  in general position and for $\mathcal{T}$ sufficiently fine,
 the integrals of the first $k$
1-forms in $V_x$ and $V_y$, $2\le k\le 5$, over $\phi(S^1)\cap T$ for each $T\in\mathcal{T}$ 
determine an $\mathcal{O}((\Delta x)^k)$-accurate approximation of $\phi(S^1)$. 
\end{proposition}
\begin{proof}
From Proposition~\ref{prop:constants}, the piecewise constant currents determine the occupied
triangles.
Suppose that the shape can be written on an occupied triangle $T$ in the form $y =g(x)$. If the derivatives
of $g$ are bounded we  use the 1-forms in $V_x$ only. Otherwise, the curve can be written in the form $x=f(y)$ where $f$ has bounded derivatives and we use the 1-forms in $V_y$ only. Without loss
of generality, we consider the first case.

The case $k=2$ is covered in Proposition~\ref{prop:constants}. 

For $k=3$, we know the intersection points of the curve with the edges of each occupied triangle and,
in addition, the value of the current $\int_U y\, \d x$ for each $U=\phi(S^1)\cap T$, $T\in\mathcal{T}$.
We take coordinates in which the shape is $y=g(x)$ 
and the $x$-intersections are at $x=0$ and $x=h$, where $h = \mathcal{O}(\Delta x)$; this
fixes the parameterization. We consider the
quadratic approximation $y=\hat g(x)$ to $g(x)$ that interpolates $g(x)$ at $x=0$ and at $x=h$ 
and has the same value of $\int_0^h y\, \d x$. This yields the approximation
$$ \hat g(x) = g(0) \left(\frac{h-x}{h}\right) + g(h)\frac{x}{h} + a_0 x(h-x)$$
where
$$ a_0 = \frac{6}{h^3}\left(\int_0^h g(x)\, \d x - \frac{g(0)+g(h)}{2}\right) .$$
At $x=s h$, $0\le s\le 1$, the approximation error is
$$ g(s h) - \hat g(s h) = \frac{h^3}{12}s(1-s)(2s-1)g'''(0) + \mathcal{O}(h^4).$$
That is, piecewise linear elements determine the shape to third order accuracy.

For $k=4$, we know, in addition, the current $\int_U  x y \, \d x$, and we choose  
a cubic approximation. There is a unique such cubic, and it yields a 4th order approximation to  $g(s h)$ with leading
order error
$$ \frac{h^4}{120} s(s-1)(5 s^2 - 5 s + 1)g''''(0).$$

For $k=5$, we know, in addition, the current $\int_U  y^2 \, \d x$, and we 
 seek a degree 4 approximant. The equations for the coefficients of the approximant 
are now nonlinear, of total degree 2. They have two real solutions. One has leading order error
less than $0.00003 h^5 |g^{(5)}(0)|$, and the other has leading order error
less than $0.003 h^5 |g^{(5)}(0)|$. This establishes 5th order accuracy for $k=5$.
\qed
\end{proof}

Note that in practice, the integrals of both $V_x$ and $V_y$ would be used, but on most triangles they do not provide independent information.

We briefly discuss several things we learn from this proposition:
\begin{enumerate}
\item The currents determine smooth shapes very accurately on sufficiently fine meshes. The errors (less than $0.008 h^3 |g^{(3)}|$ for $k=3$, $0.0006 h^4 |g^{(4)}|$ for $k=4$, and $0.00003 h^5 |g^{(5)}|$ for $k=5$), are less than twice that of Chebyshev interpolation.
\item If the finite element space $W$ has $k$ degrees of freedom per triangle (here $k=\mathcal{O}(d^2)$), the
order of approximation appears to be $\mathcal{O}(k)$. That is, the higher dimensionality
of the target manifold $N$ does not seem to be important.
\item The inherent nonlinearity of the approximation need not be an obstacle.
Despite the nonlinearity of the approximation, leading to quadratic equations when $k=5$,
the best approximant can be chosen systematically. Nevertheless, we anticipate that at very high degrees $d$ the nonlinearity may render it difficult to reconstruct an accurate approximation.
\item
The situation here is an example of a \emph{moment problem}. If we choose coordinates
in which the base of the triangle $T$ is at $y=0$, and consider the domain:
$$E_i = \{(x,y)\colon 0\le x\le h, 0\le y \le g(x)\},$$
then we are being given the moments $\iint_{E_i} x^m y^n\, \d x\d y$ (for some set of values of $m,n$)
and are asking how accurately we can reconstruct $g(x)$. When $n>1$ we have a
 nonlinear approximation problem about which, as far as we know, little is known.
\item
The moment problem considered in this section for discontinuous piecewise polynomial finite elements is nonlinear, but at least it is entirely local. We anticipate that analogous results to
Proposition \ref{prop:approx} hold for the approximation of shapes by the currents of continuous piecewise polynomial finite elements.
\end{enumerate}

\subsection{Discretization of the metric on shapes}
\label{sec:metric}
Recall that the norm of the shape metric is defined by Eqs. (\ref{eq:riesz},\ref{eq:norm}):
$$ \|[\phi]\|_{H^{-s}} := \|\beta\|_{H^s},\quad
\hbox{\textrm where\ } [\phi](\alpha)=(\beta,\alpha)_{H^s}\ \forall \alpha,\ \beta\in H^s\Lambda^1(\Omega).$$
 If $W\subset H^s\Lambda^1(\Omega)$ then both $[\phi]$ and the $H^s$ inner product
 may be restricted to $W$ to yield a finite element approximation of the representer $\beta$
 and a metric on $W^*$. That is:
 $$ \|[\phi]\|_{W^*} := \|\beta\|_{H^s},\quad
 \hbox{\textrm where\ } [\phi](\alpha) = (\beta,\alpha)_{H^s}\ \forall  \alpha,\ \beta\in W.$$
In coordinates, let $w_1,\dots,w_k$ be a basis of $W$ and let ${\mathbf b}$, ${\mathbf f}$ be 
the coordinate vectors of $\beta$ and $[\phi]$, respectively; that is, $f_i = [\phi](w_i)$ and
$\beta = \sum_{i=1}^k b_i w_i$. Then
\begin{equation}
\label{eq:Gbf}
 G{\mathbf b} = {\mathbf f},
 \end{equation}
where 
\begin{equation}
\label{eq:G}
G_{ij} = (w_i,w_j)
\end{equation}
is the matrix of the $H^s$-metric restricted to $W$. For $s=1$, $G$ is a linear combination of the mass and stiffness matrices of $W$.
Note that (\ref{eq:Gbf}), (\ref{eq:G}) amount to a standard finite element solution of the inhomogeneous
Helmholtz equation (\ref{eq:hh}).
Then we have
$$\|[\phi]\|_{W^*} = \sqrt{{\mathbf b}^T G {\mathbf b}}\,.$$
In practice, we use the Cholesky decomposition of $G$ to represent $\beta$ 
in an orthonormal basis, i.e., we compute $\tilde{\mathbf b} = G^{\frac{1}{2}}{\mathbf b}$ so that
 $$\|[\phi]\|_{W^*} = \|\tilde{\mathbf b}\|_2.$$
In this way each shape maps to a point in a standard Euclidean space $\R^{|W|}$
and standard techniques such as Principal Components Analysis can be applied.
We think of $\|\cdot\|_{W^*}$ as providing a highly compressed or approximate
geometric representation of shape space.

Although the choice $s=0$ does not make sense at the continuous level---the `representer'
would be a delta function supported on the shape, which is not in $L^2$---it does
make sense at the discrete level. An example is provided by the piecewise
constant finite elements considered in Section \ref{sec:approx}. The representer
is nonzero on the triangles that intersect the shape; it is similar to a discrete greyscale drawing
of the shape. As the currents of piecewise constant finite elements know the intersection of the shape with the edges of the mesh, they
 already provide a very sensitive discretization of the metric on shapes. 
However, it does not converge as $\Delta x\to 0$. Moreover, in this approximate metric, 
all shapes that intersect with different edges are seen as equally far away. For example,
the computed distance between two straight segments located at $x=0$ and $x=d$
is proportional to 
\begin{equation*}
\begin{cases}
d, & d < \Delta x,\\
\Delta x, & d \ge \Delta x.
\end{cases}
\end{equation*}
For these reasons we do not consider discontinuous elements further.

An example of an $H^1$ representer is shown in Figure \ref{fig:rep}, calculated
using a single square element together with polynomial currents of degree less than 10. The Helmholtz equation
(\ref{eq:hh}) smears out the shape over a length scale $\scale$ ($1/\sqrt{10}$ in this example). This
allows shapes to be sensitive to each other's positions over lengths of order $\scale$, which is 
typically many times the mesh spacing.

\begin{figure}
\centering
\includegraphics[width=7cm]{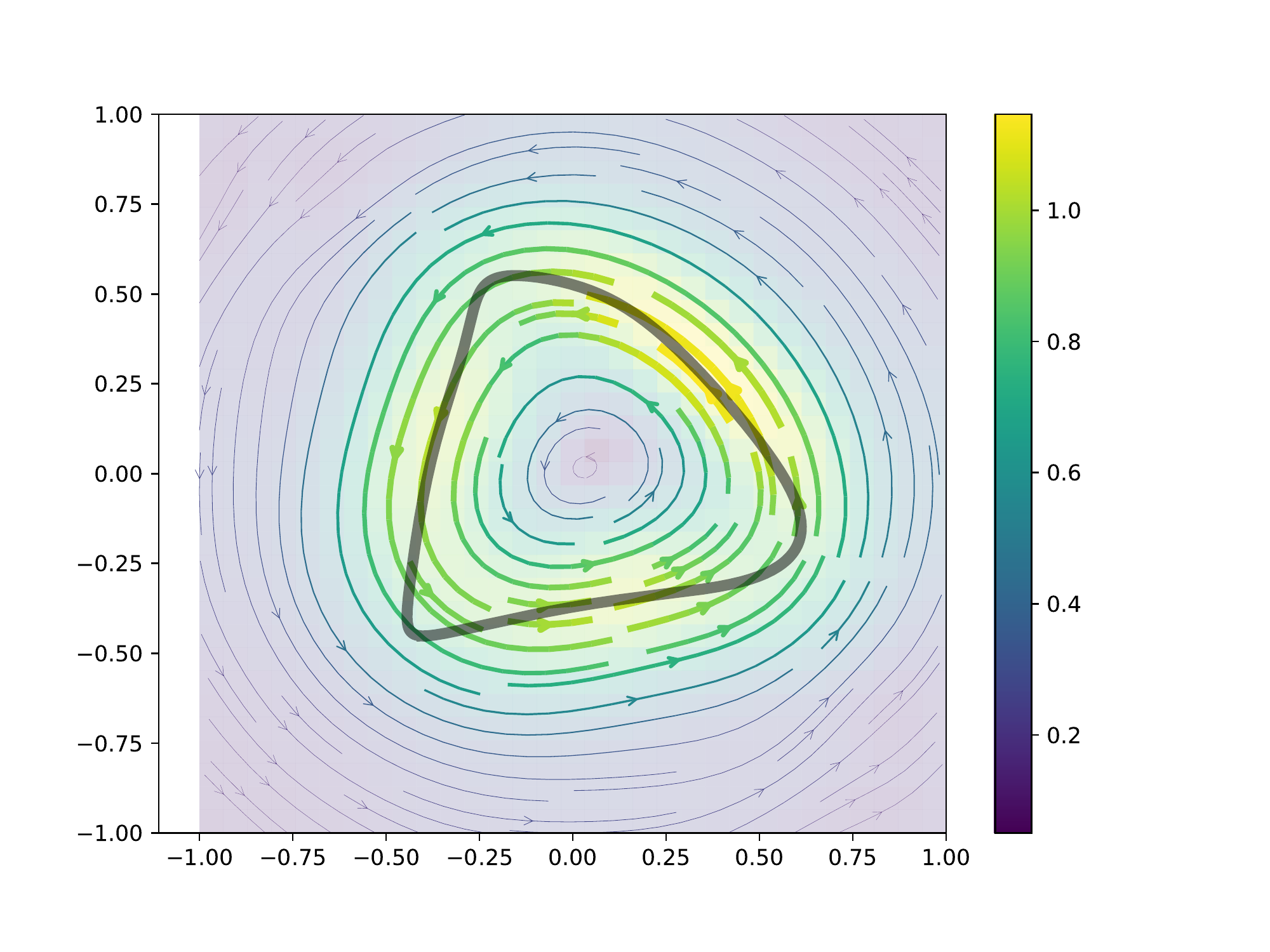}
\caption{\label{fig:rep}
The representer (blue vector field) of a randomly chosen shape (red) in $[-1,1]^2$ with respect to 
the $H^1$ metric with length scale $\scale=1/\sqrt{10}$. Here $W$ is
the space of polynomials on $[-1,1]^2$ of degree less than 10, times $\d x$ or $\d y$. 
The colour represents the size of the representer vector.
}
\end{figure}

For any $s$, one possible choice of $W$ is to take polynomials up to some degree $N$.
This gives a spectral method. However, we do not expect spectral accuracy
because the representers are not smooth, they are only in $H^s$. 
As in finite element solutions of PDEs, the mesh size and order $p$ of the elements
can be adjusted depending on the application. But unlike PDEs, in most
shape applications it is not necessary to have small errors with respect to the continuous problem. Rather, 
each choice of $W$ 
provides a different approximation or description of shape space. In some applications
$W$ may be relatively low-dimensional.

For $s=1$, we can take any triangulation of $\Omega$ with $W$ the continuous piecewise polynomials of degree $d\ge 1$; these lie in $H^1$ so the above construction applies directly.
For $s>1$ we use the same elements,
and form the same metric $G$ associated with $s=1$, but determine ${\mathbf{b}}$ by solving
\begin{equation}
\label{eq:Gs}
G^s{\mathbf{b}} = {\mathbf{f}}.
\end{equation}
Such representers satisfy higher-order Helmholtz equations
with slightly different boundary conditions than those implied by Eqs.\ (\ref{eq:riesz}) and (\ref{eq:norm}).
In practice, this difference is immaterial
as long as the curve is located away from the boundaries,
since the representer tends to zero far from the curve.
Moreover, this approximation error is outweighed by the  ease of solving (\ref{eq:Gs}) and of using standard finite elements.

\section{Examples}
\label{sec:examples}

We have implemented this approach to currents using finite elements in Python using FEniCS and Dolfin~\cite{Logg12,Alnaes15}. 
The implementation is available at \cite{femshape}.
We define a rectangular mesh with continuous Galerkin elements of predefined order. To compute the invariants, we use the \texttt{tree.compute\_entity\_collisions()} function to identify which mesh cells intersect with the curve. We then evaluate the basis functions of those cells and update the computation of the currents as the sum of the basis elements of intersecting mesh elements weighted by the size of the cell. Solving for the representer in the relevant norm is based on a matrix solve $\mathbf{G} b_x = f_x$, where $\mathbf{G}$ is the tensor representation of the norm (see equation~(\ref{eq:G})), $b_x$ is a function on the finite element space, and $f_x$ is the vector of $x$ currents (i.e.,
an integral of $w_i(x)\, \d x$ over the curve), and similarly for $b_y$ and $f_y$. We used a length scale of $\scale=1/\sqrt{10}$ throughout, so the $H^{-1}$ norm that we used was computed as \texttt{m = 1./10*inner(grad(u),grad(v))*\d x () + u*v*\d x ()}, where \texttt{u} and \texttt{v} were trial and test functions on the function space, respectively. In order to compute the $H^{-2}$ norm it is necessary to perform a second solve, based on the matrix built from \texttt{x*v*\d x ()}.

We present a series of experiments demonstrating the use of our implementation on a variety of test cases. The first few examples are chosen to demonstrate the robustness of the approach. They demonstrate the reparameterization of the shape, its numerical stability of the method in both norms, and robustness with curves that lose differentiability, for example at corners. These examples are followed by a set of experiments demonstrating that the representation of the curves that are produced is consistent with human perceptions of the shapes.

\begin{example}
\label{ex:reparam}
  In the first experiment we next demonstrate that the representation in the $H^{-1}$ and $H^{-2}$ norms is invariant to reparameterization of the curves. The curve is a bowtie shape discretized with 512 points. The equally-spaced positions of these points were then perturbed by a Gaussian random variable of standard deviation 0.1 (recall that the radius of the circle was 0.5) and the points resorted into monotonically increasing order of arclength. \autoref{fig:reparam} shows the representers for the unperturbed shapes on the left, and the perturbed ones on the right using piecewise linear finite elements and a meshsize of $10 \times 10$. Visually, there seems to be no difference between them. The difference in the computed norms between the original shape and the reparameterized one was of the order of $10^{-4}$ for both metrics.

\begin{figure}
\centering
\includegraphics[width=\textwidth]{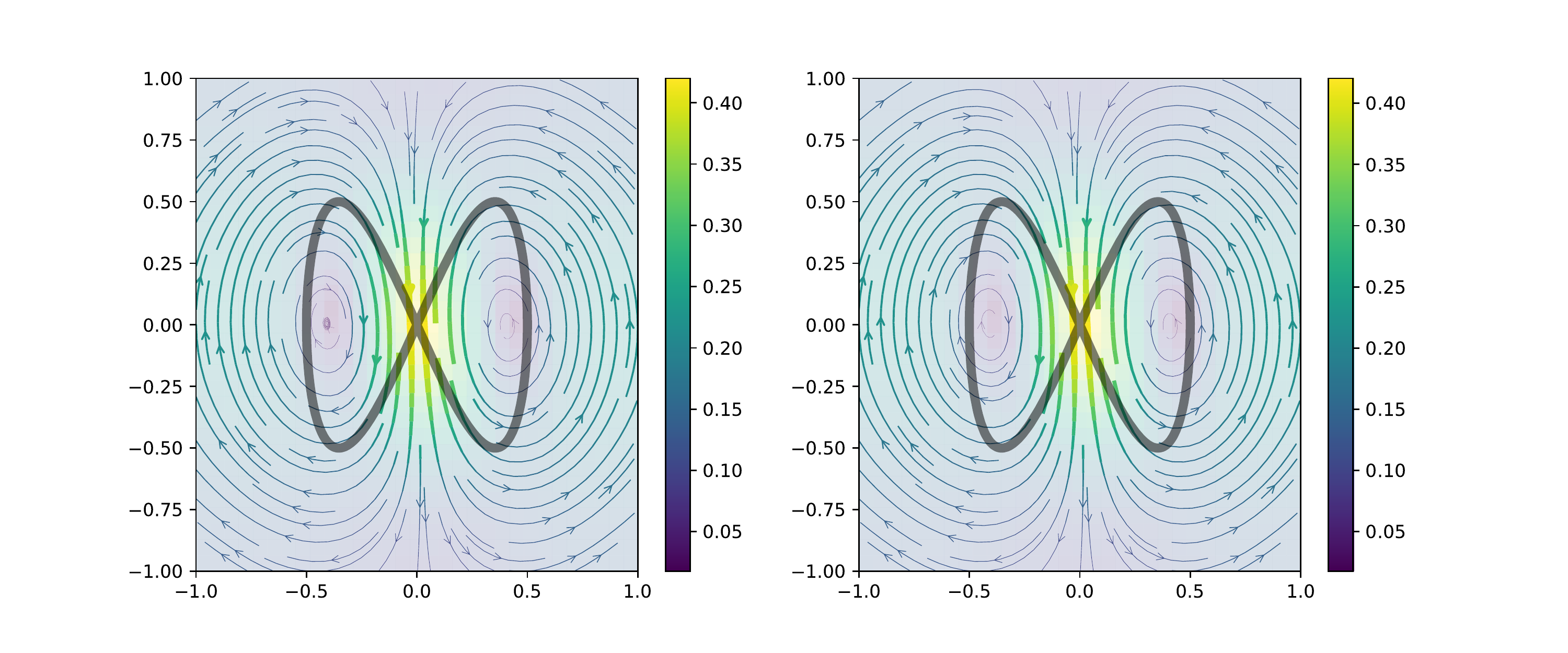}
\caption{\label{fig:reparam}
\autoref{ex:reparam}:
The representers for a bowtie shape with equally spaced points (\emph{left}) and under a random reparameterization (\emph{right}). Visually there seems to be no difference between the representers, and the numerical differences are around $10^{-4}$ with both metrics.
}
\end{figure}
\end{example}

\begin{example}
\label{ex:errorH1H2}
We examine the convergence of the discretized metric with respect to the
meshsize and the order of the finite elements. A circle of radius 0.5 in the domain $\Omega = [-1,1]^2$ discretized with 5000 equally spaced points, so that quadrature errors are negligible. The grid is a uniform triangular mesh on an $M\times M$ grid. 
Using elements of order 1 (piecewise linear) up to order 4, we computed the difference in each of the $H^1$ and $H^2$ norms as the meshsize $M$ was increased in powers of 2 from 1 up to 128, i.e., we compute $\|[\phi]\|_{H^{-s},2M} - \|[\phi]\|_{H^{-s},M}$. The results are shown in \autoref{fig:errorH1H2}.  Reference lines illustrating convergence of order $1$ for $H^{-1}$ and order 2.5 for $H^{-2}$ are also shown.
The improved convergence of the $H^{-2}$ metric is clear; the benefit of increasing the order of the finite elements themselves is less clear. 

\begin{figure}
\centering
\includegraphics[width=.45\textwidth]{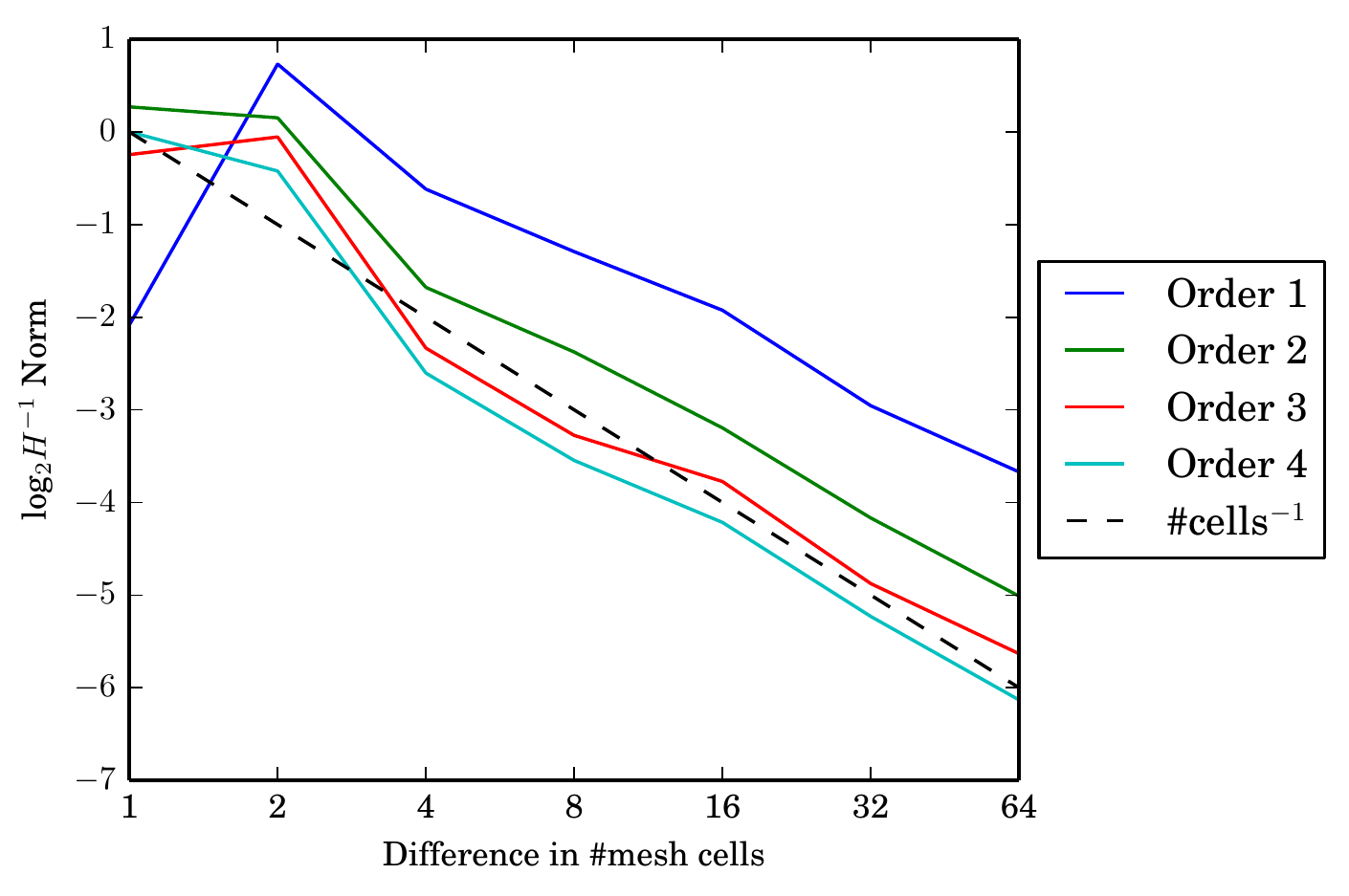}
\includegraphics[width=.45\textwidth]{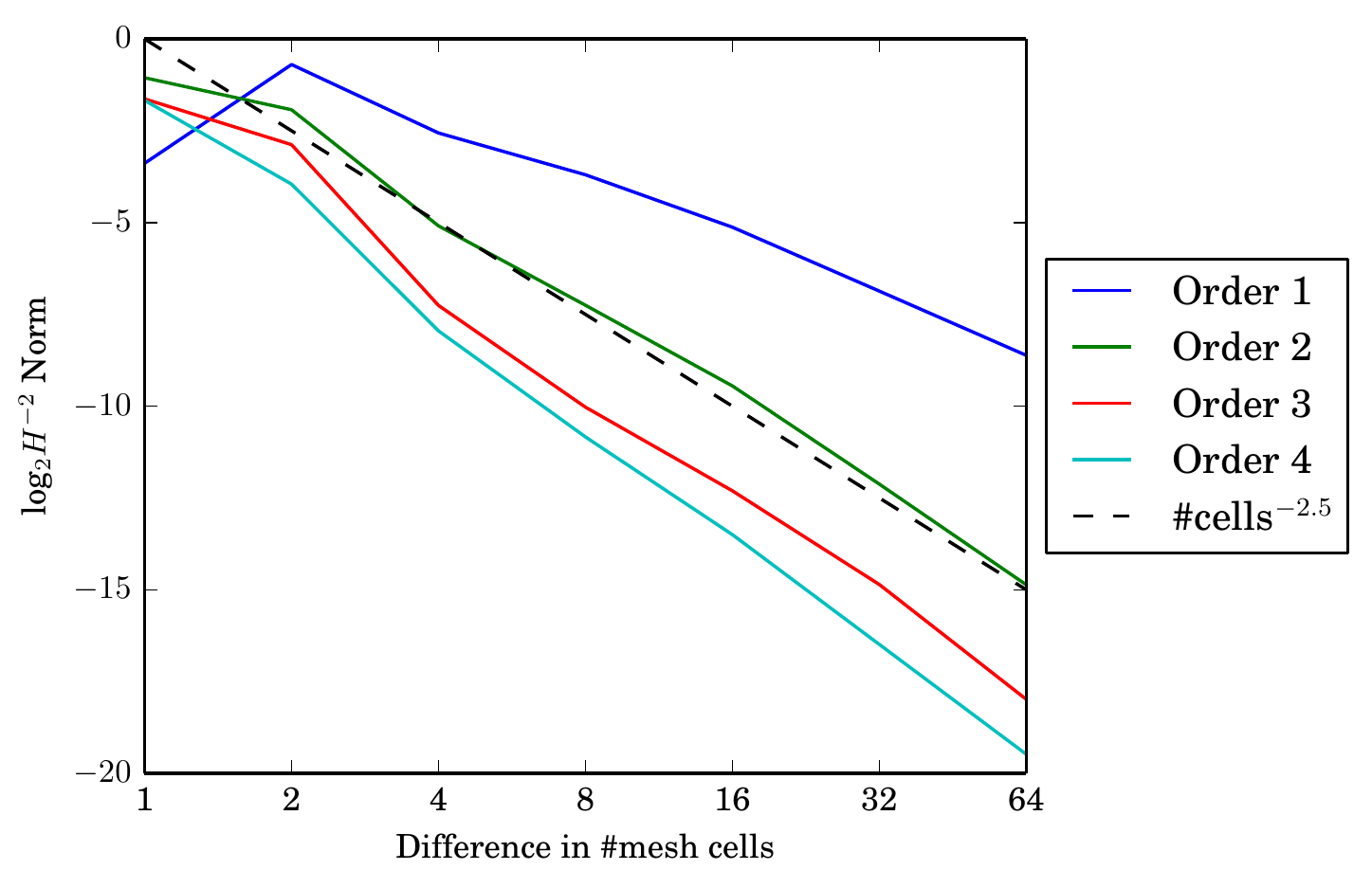}
\caption{\label{fig:errorH1H2}
\autoref{ex:errorH1H2}:
Convergence of the discrete metric as a function of the meshsize and element order.
The log difference between successive approximations of the $H^{-1}$ norm (\emph{left}) and $H^{-2}$ norm (\emph{right}) is shown as the meshsize increases in powers of 2 for four different orders of finite elements. Reference lines are provided in each plot to illustrate the  order of convergence. }
\end{figure}
\end{example}

\begin{example}
\label{ex:wiggly}
In this example we study the sensitivity of the norms to small wiggly perturbations. 
The base shape is the same circle of radius 0.5, and the domain remains $\Omega = [-1,1]^2$. The circle
is perturbed by scaling its radius by a factor $1+\epsilon \cos(\omega \theta)$; see the illustration in \autoref{fig:wiggly}.
We take 5000 points on the shapes and compute the distances using piecewise linear finite elements 
and apply Richardson extrapolation to the results
from meshsizes $M=80$, $M=160$, and $M=320$ to ensure that finite element discretization errors
are negligible. The results for the two norms are shown in the table in \autoref{fig:wiggly}. As expected 
from the analysis of nearby straight sections in Section \ref{sec:metric}, for 
fixed frequency $\omega$ the distances are $\O(\epsilon^{1/2})$ for the $H^{-1}$ metric
and $\O(\epsilon)$ for the $H^{-2}$ metric. However, it appears that overall the distances
are independent of frequency for $H^{-1}$ and $\O(\omega^{-1/2})$ as $\omega\to\infty$ for
$H^{-2}$. Overall, the $H^{-2}$ metric is much less sensitive to small wiggles, 
especially small high-frequency wiggles, than the $H^{-1}$ metric.
\begin{figure}

\centering
\includegraphics[height=6cm]{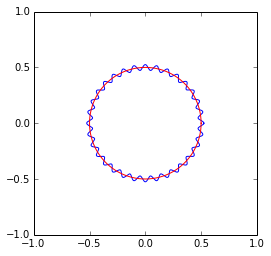}
\hspace{2mm}
\includegraphics[height=6cm]{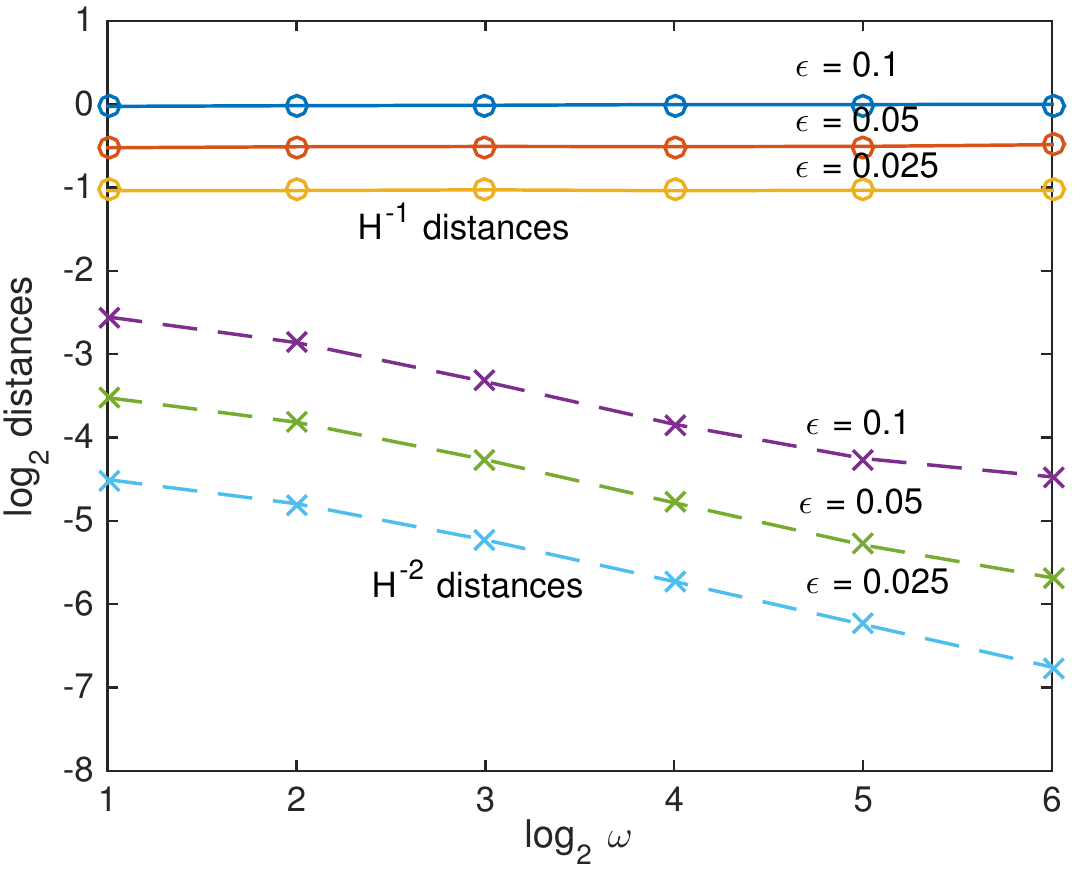}\\[2mm]
\begin{tabular}{| r | rrr || rrr |}
\hline
$\omega$ & \multicolumn{3}{|c||}{$H^{-1}$} & \multicolumn{3}{|c|}{$H^{-2}$} \\
\hline
 & $\epsilon=0.1$ & $\epsilon=0.05$ &$\epsilon=0.025$ &$\epsilon=0.1$ & $\epsilon=0.05$ &$\epsilon=0.025$ \\
 \hline
2   &0.9817  &  0.6959 &  0.4868 &   0.1704 &   0.0871 &   0.0440 \\
4  & 0.9876  &  0.7017  &  0.4875 &  0.1376  &  0.0709 &   0.0360\\
8  & 0.9905   & 0.7034   & 0.4903  &  0.0994  &  0.0520 &   0.0267\\
16 &  0.9969  &  0.7027  &  0.4868 &  0.0698  &  0.0363&    0.0189\\
 32 & 0.9967   & 0.7037   & 0.4886  & 0.0525  &  0.0256&    0.0132\\
 64 & 0.9991  &  0.7140   & 0.4881  & 0.0450  &  0.0195&    0.0093\\
\hline
\end{tabular}

\caption{\label{fig:wiggly}
  \autoref{ex:wiggly}:
  The reference shape (red circle) and perturbation (wiggly blue curve) shown
for amplitude $\epsilon=0.05$ and frequency $\omega=32$; the computed
distances between the reference and perturbed shapes is shown at right, and tabulated below,
in the $H^{-1}$ and $H^{-2}$ norms.
The norm perturbation due to the wiggles appears to be $\mathcal{O}(\epsilon^{-1/2})$ for
the $H^{-1}$ norm and $\mathcal{O}(\epsilon\omega^{-1/2})$ in the $H^{-2}$ norm. 
}
\end{figure}
\end{example}

\begin{example}
  \label{ex:rounding}
  In this experiment we consider a family of shapes, the supercircles $x^r + y^r = (1/2)^r$. shapes, together with their representers, are shown in \autoref{fig:rounding} with 512 points on the curves,  meshsize $M=80$ and piecewise linear elements, and $r=2^1,2^{1.5},2^2,2^{2.5}$.
  Note that although the different between the \emph{curves} is apparent, it is much harder to see a difference between the \emph{representers}.
Nevertheless, the norms do see a difference, with the $H^{-1}$ norm decreasing from 4.67 to 3.61  over the four curves shown 
in the Figure. The $H^{-2}$ norm decreases from 1.86 to 1.26. 

\begin{figure}
\centering
\includegraphics[width=\textwidth]{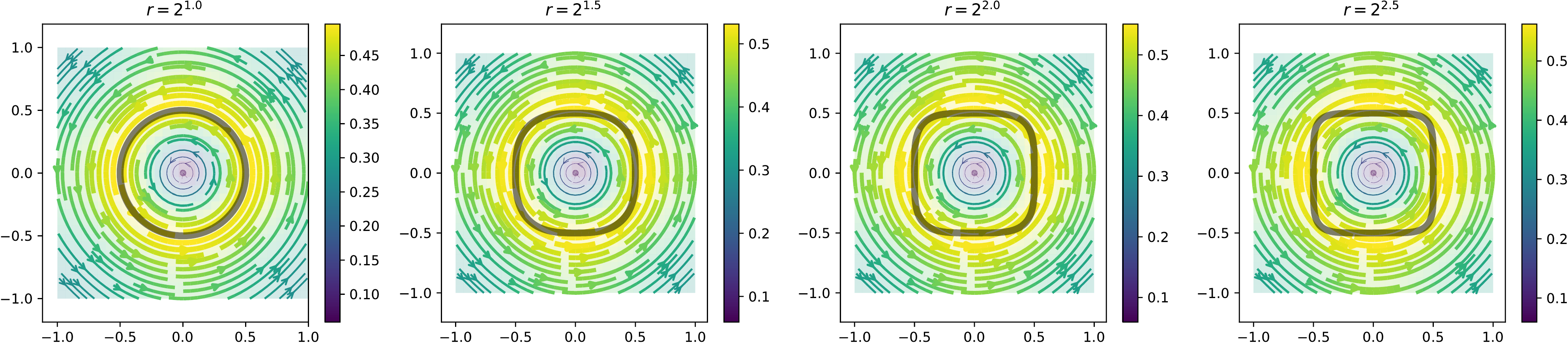}
\caption{\label{fig:rounding}
\autoref{ex:rounding}:
The representers for a set of shapes showing a circle deforming to get progressively more `square-like'.
With the chosen scale $\scale = 1/\sqrt{10}$ on a $1 \times 1$ grid, the visual differences between the representers for the various shapes are very small.
}
\end{figure}

\end{example}



\begin{example}
  \label{ex:rounded}
  In this example we study the $H^{-1}$ geometry of the family of supercircles. We take the curves $x^r + y^r = (1/2)^r$ for $0<r<\infty$. This family ranges from a cross (equivalent to a null shape for currents), through an astroid, a circle, to a square. 
We compute the shapes for 13 exponents $r=2^j$, $j=-3, -2.5, \dots,3$. 
We use 512 points on each curve, a meshsize of $M=10$ and piecewise linear elements.
Thus $W$ is a normed vector space of dimension $2\cdot(M+1)^2=242$. It is
transformed to a standard Euclidean $\R^{242}$ using the Cholesky factorization
described in Section \ref{sec:metric}. The 12 data points in $\R^{242}$ are
projected to $\R^2$ using standard PCA.
 The results are shown in \autoref{fig:rounding2}. The bunching up of the points at either end is clear. The induced geometry of the family can be seen in the curvature of the
 family.

\begin{figure}
\centering
\includegraphics[width=.5\textwidth]{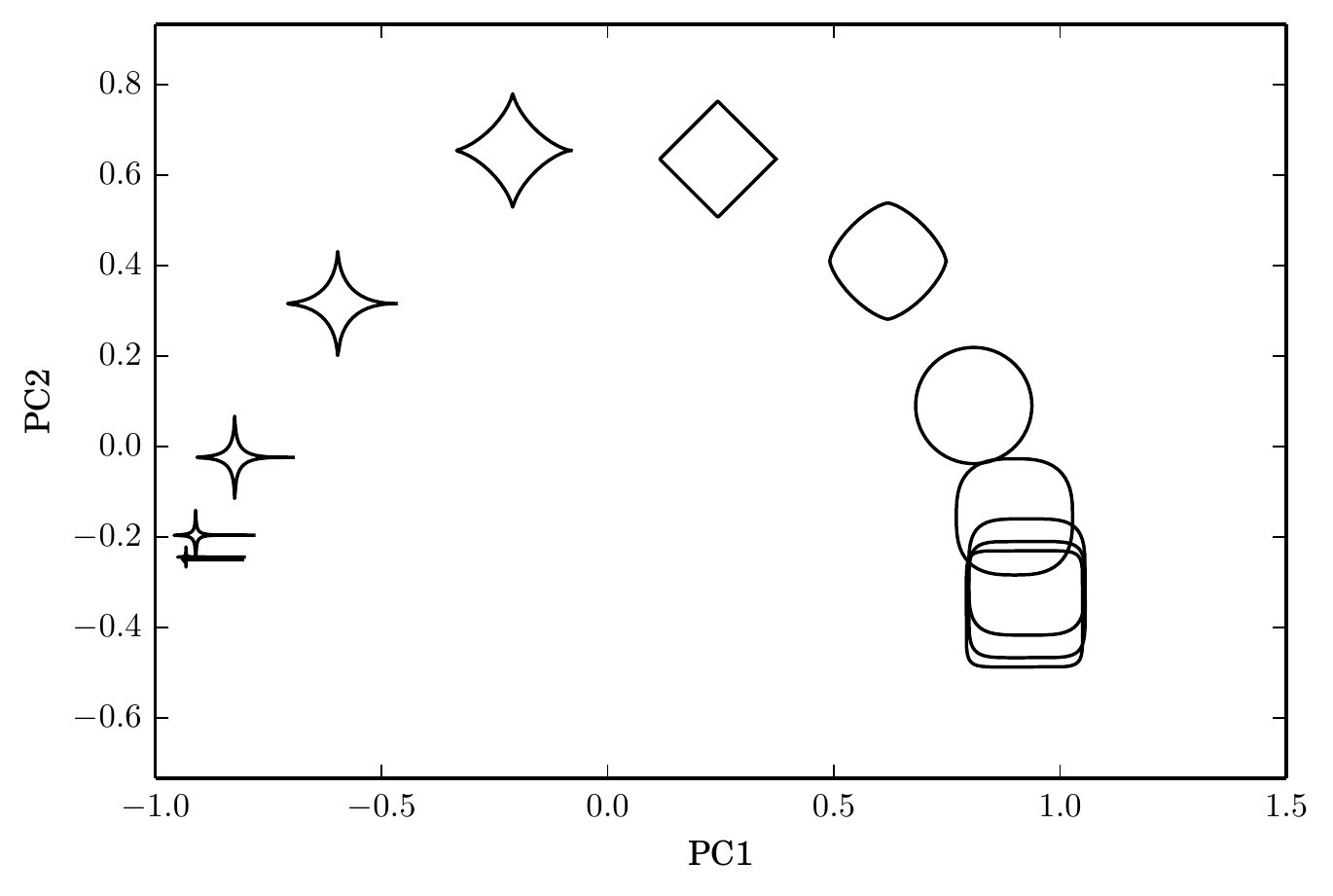}
\caption{\label{fig:rounding2}
\autoref{ex:rounded}:
2D embedding of a set of simple shapes that have a strong dependence on a single parameter by using the first two principal components.
The curve of shapes seems to have constant curvature, but we do not know why it is so.
}
\end{figure}
\end{example}

\begin{example}
  \label{ex:diffsig}
  In the previous example, the dataset was intrinsically 1-dimensional. We now consider a high-dimensional dataset. We generate 32 random smooth shapes using random Fourier coefficients and then compare them in the $H^{-1}$ metric using  currents $\int x^m y^n\, \d x\, \d y$ for $0\le m+n<10$. The Fourier coefficients $\tilde z_k$ of the shapes have $\tilde z_0=0$, $\tilde z_1=0.5$ $\tilde z_{-1}=0$ (so that they are all roughly centred and of the same size), with  $\tilde z_k$ for $2\le k\le 6$ being independent random normal variables with standard deviation $1/(1+|k|^3)$. Thus the dataset is 10-dimensional. We then perform an optimization step that computes the planar embedding of the shape currents whose Euclidean distance
matrix best approximates the $H^{-1}$ distance matrix, using the best planar subspace of $W$ as computed by PCA as the initial condition for the optimization.
The results are shown in \autoref{fig:randN10}.
The mean distance error of the embedding is 0.06. 
This gives a pictorial representation of a high-dimensional data set lying within
an infinite-dimensional nonlinear shape space. The triangular shapes appear on the 
edges, sorted by orientation, while the squarish shapes appear near the centre. 
Obviously similar shapes, such as 5 and 16, and 2 and 4, are placed very close
together.

\begin{figure}
\centering
\includegraphics[width=12cm]{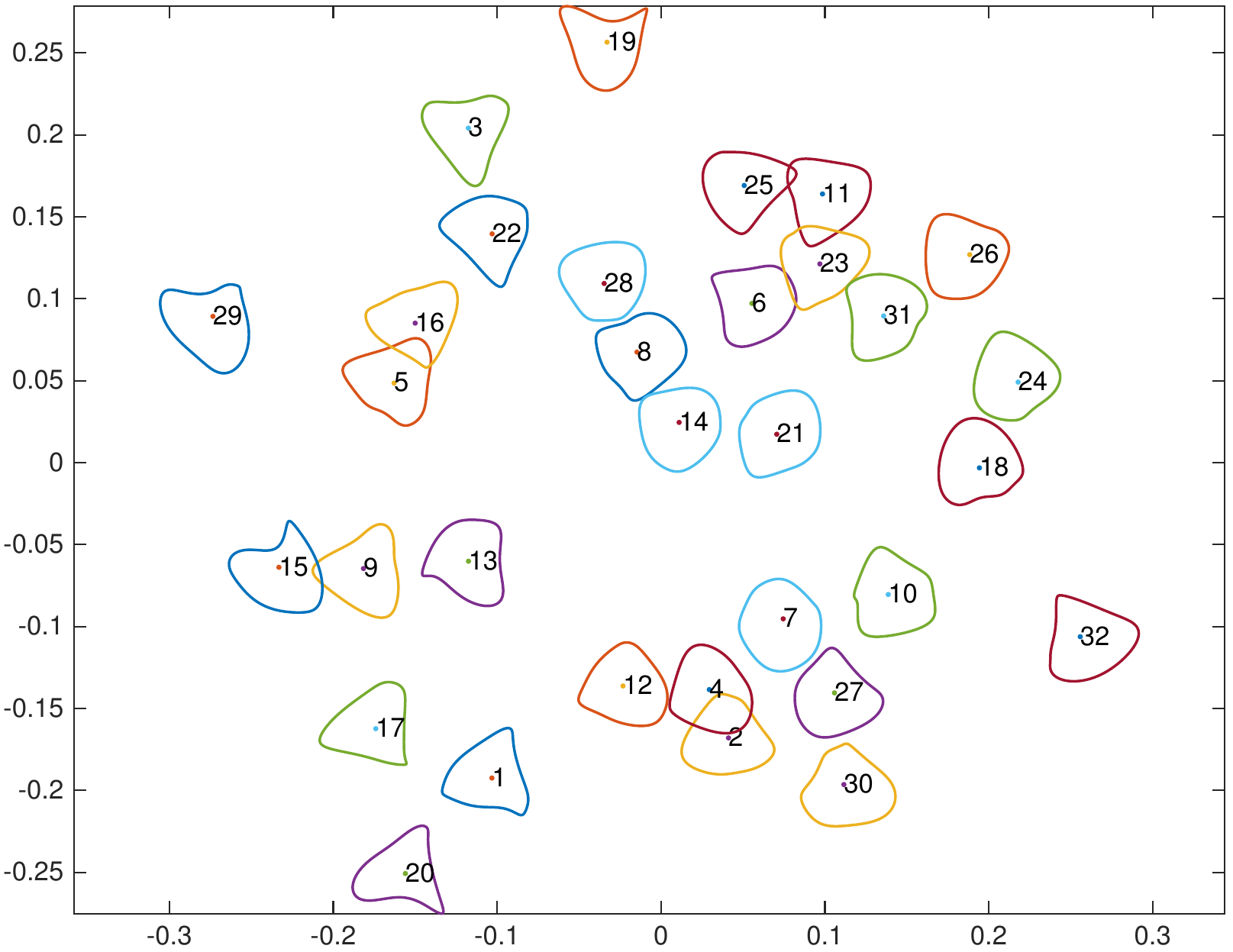}
\caption{\label{fig:randN10}
  \autoref{ex:diffsig}:
Here, 32 random shapes are created and compared using currents of order $N\le 10$. 
The planar embedding of the shape currents that best preserves their pairwise distances is shown.
One can see that shapes are grouped together according to their general shapes (triangular or square) and orientation.
As the underlying distance is robust with respect to noise, apparently different shapes such as 15, 9 and 13 are grouped together, but their shape (triangular) and orientation is roughly the same.
}
\end{figure}
\end{example}

\begin{example}
  \label{ex:2dembedeuc}
In order to compare our approach to the differential signature method, we took the same shapes as in \autoref{ex:diffsig} and compared them modulo the special Euclidean group of the plane by computing their Euclidean differential signatures $(\kappa,\kappa_s)$ (where $\kappa$ is the Euclidean curvature and $s$ is arclength) using the Euclidean-invariant finite differences introduced by Calabi et al. \cite{calabi1998differential} and corrected by Boutin \cite[Eq. (6)]{boutin2000numerically}. The signatures are mapped into $[-1,1]^2$ as $(0.8 \arctan(\kappa/3),0.8\arctan(\kappa_s/150))$; this mapping controls the weighting of the extreme features of the signatures at which $\kappa$ and (especially) $\kappa_s$) are large. The currents using moments for order $<10$ are computed as for \autoref{fig:randN10}, and the currents embedded in the plane to preserve as best as possible their pairwise distances by performing a least-squares optimisation in this space to minimise the 
distance between the pairwise distances between the points in the original space and in this 2D version.

There are some interesting differences between the two embeddings, such as the tight clustering of 14, 24, 26, 28 in this Figure, of which only 24 and 26 are close in \autoref{fig:randN10}. Looking at the position of shape 10 in both Figures, it is also clear that the differential signature is dominated by the slight kink on the right of the shape (see also shape 13). 

\begin{figure}
\centering
\includegraphics[width=10cm]{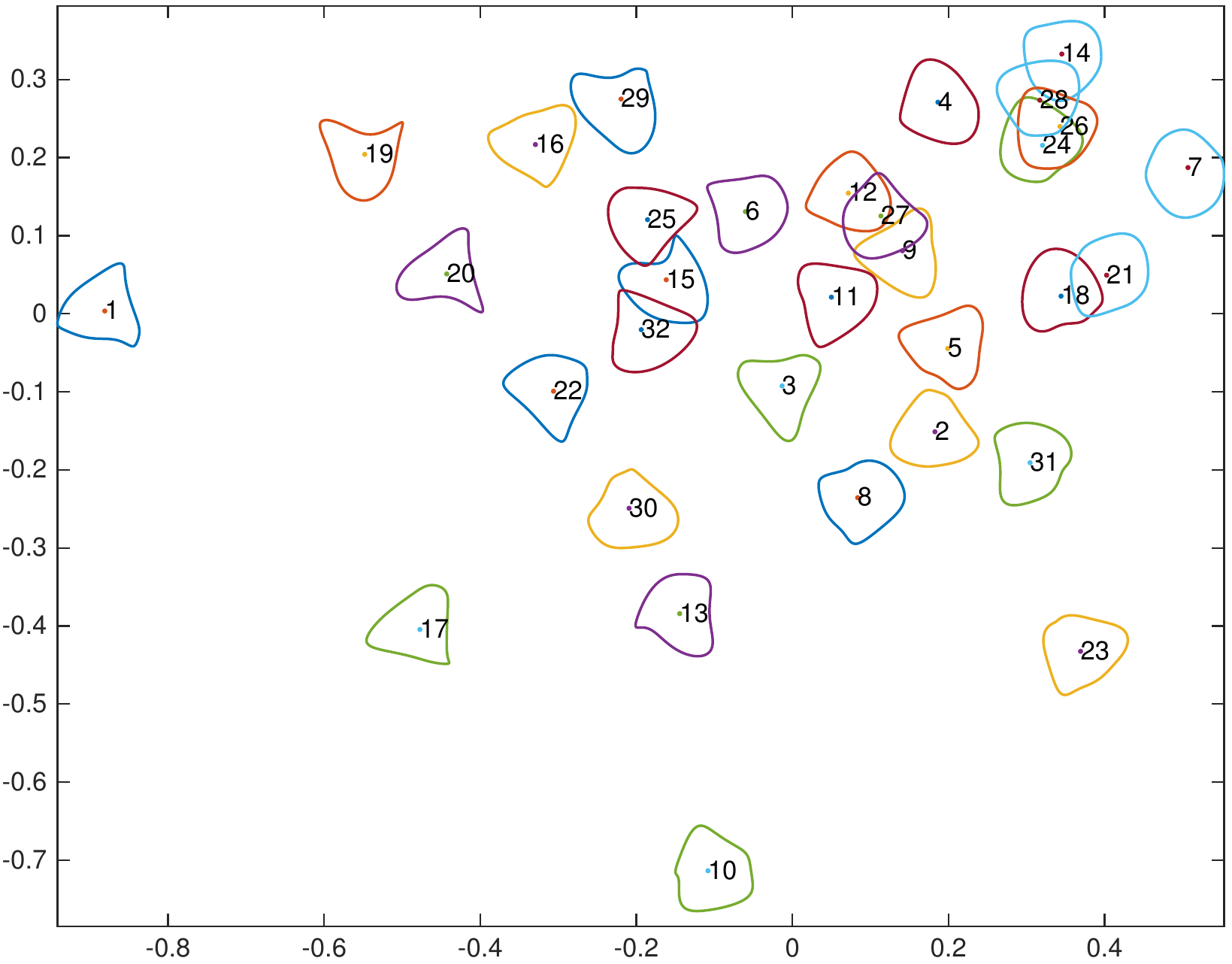}\\
\includegraphics[width=12cm]{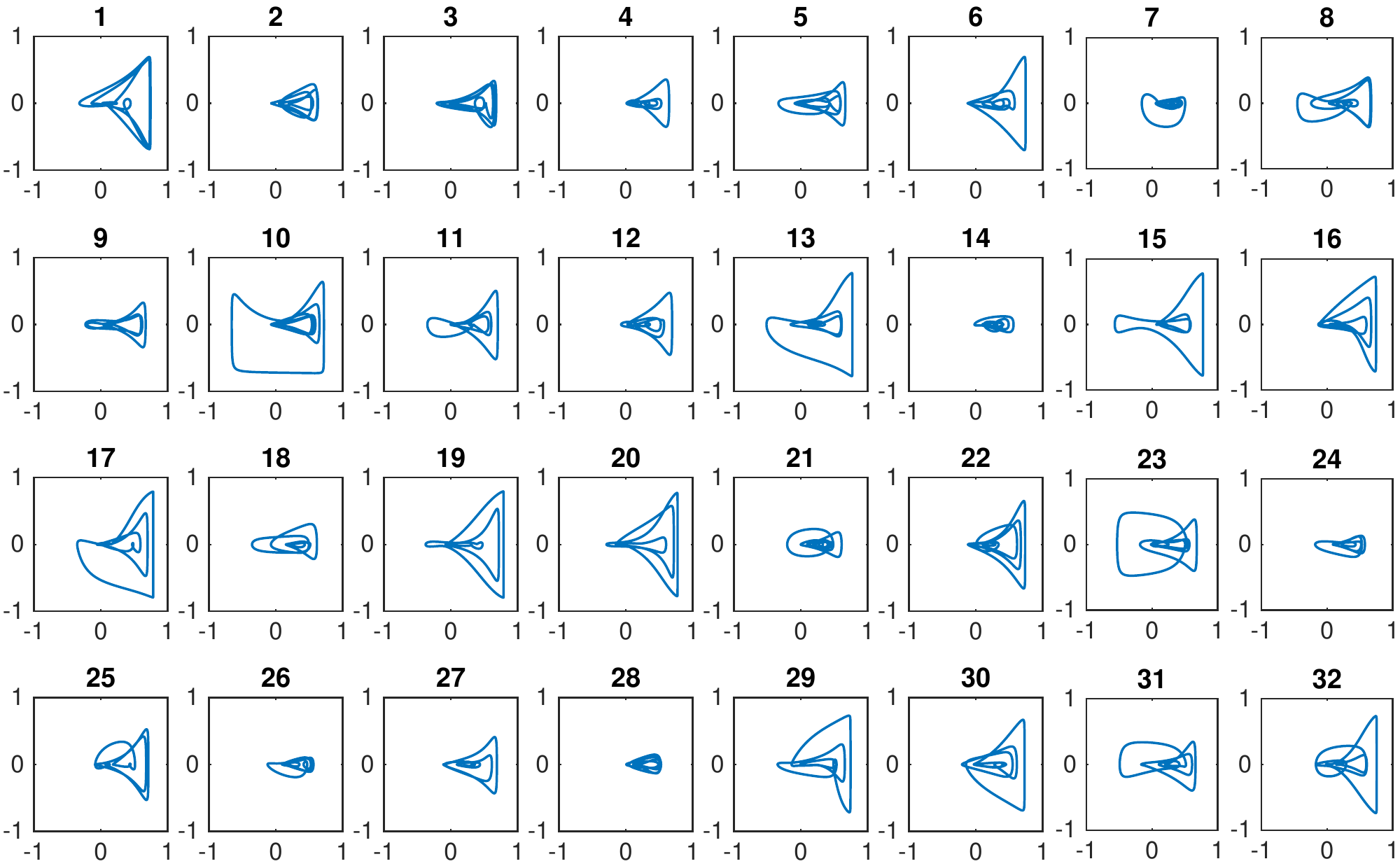}
\caption{\label{fig:2dembedeuc}
  \autoref{ex:2dembedeuc}
Here, the same 32 shapes as in \autoref{fig:randN10} are compared modulo the special Euclidean group of the 
plane (top). The differential signature of each shape is plotted below.
}
\end{figure}
\end{example}

\begin{example}
  \label{ex:fish}
  Our next example (\autoref{fig:fish.pdf}) is a family of complicated shapes that have a linear dependence on a parameter,  $\phi(t;a) = \phi_1(t) + a \phi_2(t)$. As the shapes are relatively complicated and the differences between them relatively small, the currents of the shapes (again, moments of order $<10$) embed extremely well in a plane: mean distance error $<10^{-3}$.
However, in this embedding in the plane, significant geometry is seen within the family, with speed along the curve decelerating, changing direction, and accelerating. This example
and the previous one illustrate that the method of finite element currents copes with 
complicated immersions.

\begin{figure}
\centering
\includegraphics[width=8cm]{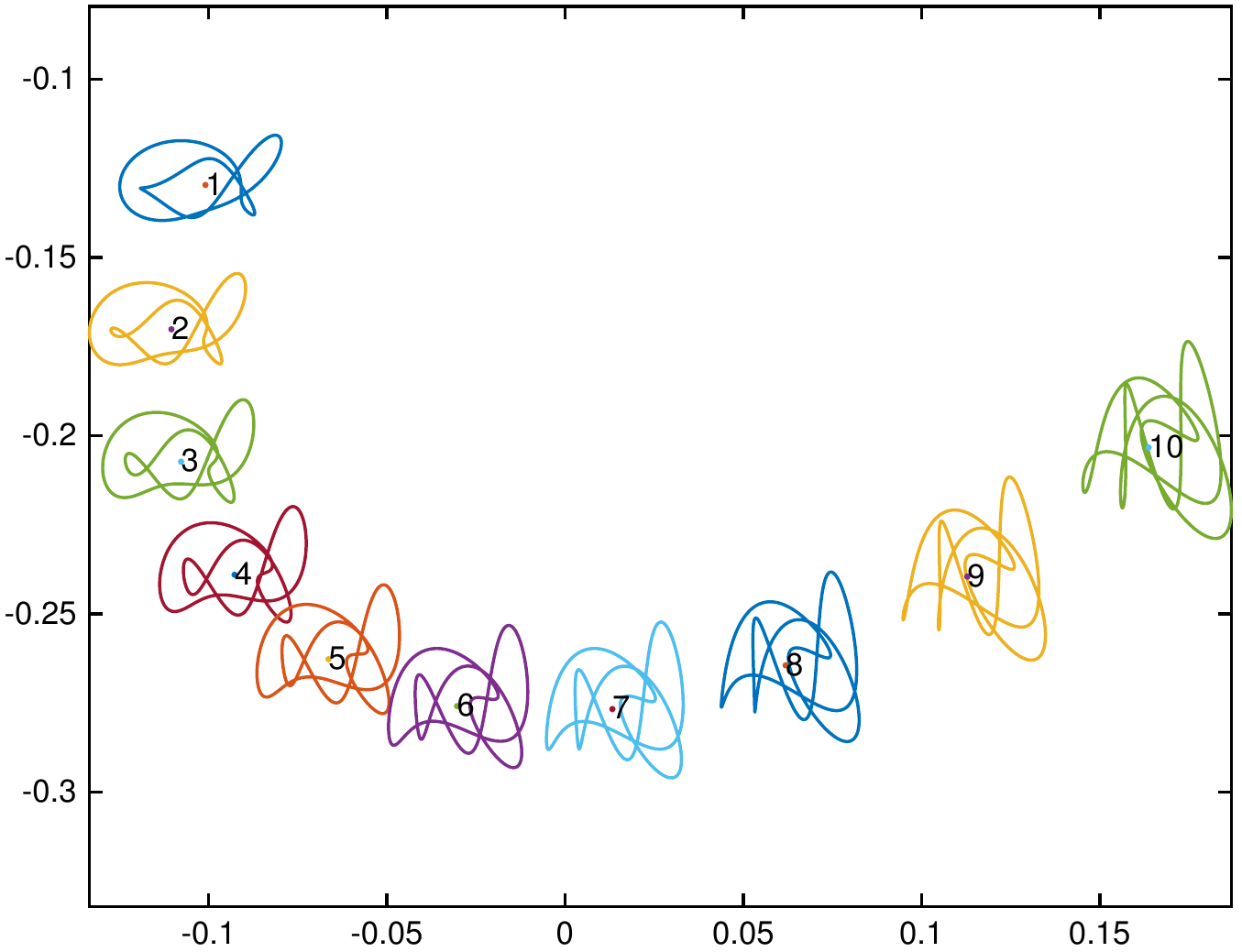}
\caption{\label{fig:fish.pdf}
  \autoref{ex:fish}:
A family of shapes is chosen with linear dependence on a parameter. They embed extremely well into the plane, largely forming a curve that reflects the dependence upon the parameter.}
\end{figure}
\end{example}

\begin{example}
\label{ex:final}
In our final example, we investigate how currents can be used to separate populations of shapes.
We generate a set of shapes using six Fourier components as previously,
but
we choose different values for the 3rd and 4th Fourier components and add relatively small noise to them.
The resulting shapes thus come from three different closely-related sets.
We generate a small dataset comprising ten examples from each of these sets.
These are shown in \autoref{fig:lasttestdata} in the three sets; it can be seen that it is hard to distinguish between the examples. 

We then apply PCA to the currents and use PCA to project the shapes into 2D. \autoref{fig:lasttestPCA} shows the results of this PCA with a few examples of each shape plotted. The three classes are clearly visible in this embedding. 


\begin{figure}
\centering
\includegraphics[width=.75\textwidth]{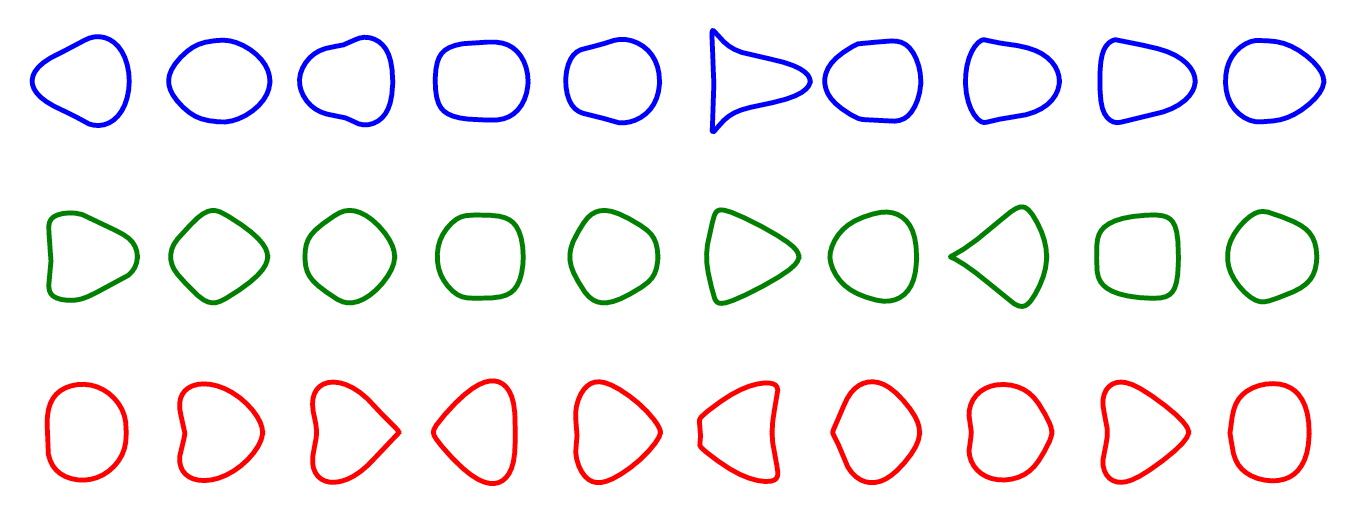}
\caption{\label{fig:lasttestdata}
  \autoref{ex:final}:
Data drawn from three slightly different sets, where the difference is in the third and fourth Fourier components. It is hard to visually distinguish the three shapes.
}
\end{figure}

\begin{figure}
\centering
\includegraphics[width=12cm]{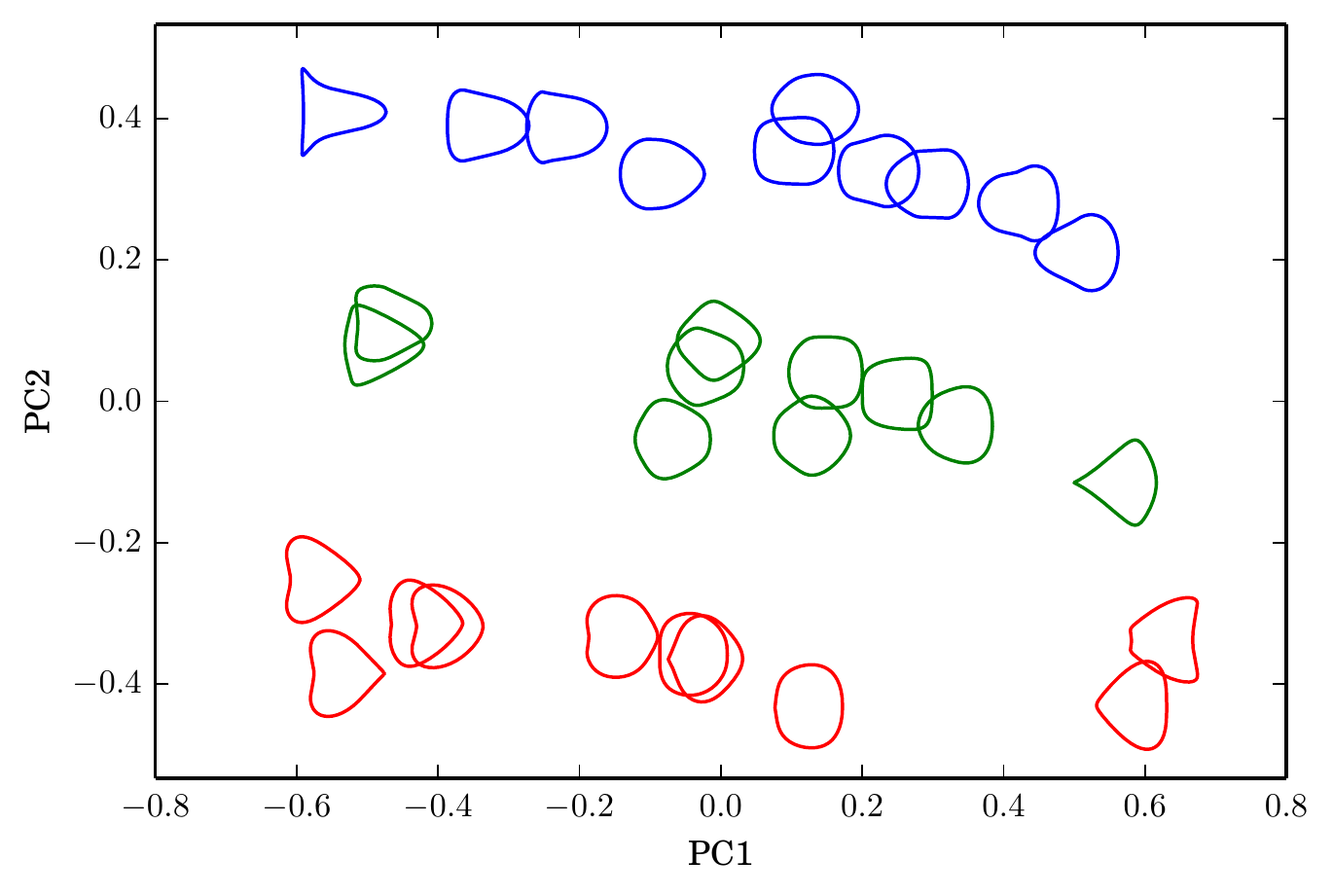}
\caption{\label{fig:lasttestPCA}
  \autoref{ex:final}:
The 2D embedding of the dataset using the first two principal components. The first principal component completely fails to separate the data, but the second does it perfectly.
}
\end{figure}


\end{example}

%
%


\bibliographystyle{amsplainnat} 
\bibliography{shape_currents}


\end{document}